\documentclass[11pt,letterpaper]{article}
\usepackage{geometry}
\usepackage{graphicx}
\usepackage{float}

\usepackage{amssymb,amsthm,amsmath,bbm,MnSymbol}
\usepackage{mathtools}
\usepackage{xfrac}
\usepackage{comment}
\usepackage{adjustbox}
\usepackage{tikz, tikz-cd}
\usepackage{enumitem}

\definecolor{darkblue}{rgb}{0,0,0.4} 
 \usepackage[colorlinks=true, citecolor=darkblue, filecolor=darkblue, linkcolor=darkblue,urlcolor=darkblue]{hyperref}

\usepackage[backend=biber, style=alphabetic, doi=false, url=false, giveninits=false, ,backref=true]{biblatex}
\DefineBibliographyStrings{english}{   backrefpage={See p.},  backrefpages={See pp.} }
\usepackage{pstricks}
\usepackage{titlesec}

\usepackage{subcaption}
\usepackage{caption}
\usepackage{import}
\usepackage{pdfpages}
\usepackage{transparent}

\tikzcdset{scale cd/.style={every label/.append style={scale=#1},cells={nodes={scale=#1}}}}

\newtheorem{IntroTheorem}{Theorem}
\newtheorem*{Theorem*}{Theorem}

\numberwithin{equation}{section}

\newtheorem{Theorem}{Theorem}[section]
\newtheorem{Lemma}[Theorem]{Lemma}
\newtheorem{Proposition}[Theorem]{Proposition}
\newtheorem{Corollary}[Theorem]{Corollary}

\theoremstyle{definition}
\newtheorem{Definition}[Theorem]{Definition}
\newtheorem{Example}[Theorem]{Example}
\theoremstyle{remark}
\newtheorem{Remark}[Theorem]{Remark}

\numberwithin{figure}{section}
\numberwithin{table}{section}

\begin{document}
\title{Obstruction to Equivariant Ribbon Concordance}
\author{Siavash Jafarizadeh}
\date{August 2025}

\maketitle
\begin{abstract}
	A periodic link, is link in $S^3$ with action of $\mathbb{Z}_p$ by rotation with $2\pi/p$ around a fixed unknot $U$. The equivariant 
    Khovanov homology of periodic links has been studied in \cite{BP17}. We prove that the equivariant Khovanov homology for 
    periodic links is functorial under equivariant cobordisms. Furthere more, we show that equivariant ribbon concordances 
    induce a split injection on equivariant Khovanov homology.
\end{abstract}

\section{Introduction}\label{Sec:Introduction}

For a prime integer $p$, we can consider the action of $\mathbb{Z}_p$ (the cyclic group of order $p$) 
on $\mathbb{R}^3$ by rotation by $\frac{2\pi}{p}$ around the $z$-axis. As this action is proper with fixed-point set the $z$-axis, 
it induces an action on $S^3=\mathbb{R}^3\cup\infty$ with the fixed points the unknot $U=z\text{-axis}\cup \infty$. 
Throughout this manuscript, we denote the $z$-axis in $\mathbb{R}^3$ by $\boldsymbol{z}$ and the origin in $\mathbb{R}^n$ by 
$\boldsymbol{0}$. A link $L\subset \mathbb{R}^3\backslash\boldsymbol{z} =S^3\backslash U$ is called $p$-\emph{periodic} if it is invariant under the action of $\mathbb{Z}_p$.

In this work, we view $p$ as fixed. Hence, when it will not cause confusion we drop the $p$ from the notation. 
Also, we denote this action by $\theta$. The action of $\mathbb{Z}_p$ on $S^3$ can be extended to 
$S^3\times [0,1]$, if we let the action be trivial on $[0,1]$. 
Let us denote the extended action by $\tilde{\theta}$. Then $\tilde{\theta}$ has fixed annulus $\tilde{U}=U\times [0,1]$.
By a \emph{cobordism} between links $L_0$ and $L_1$ we refer to a smoothly embedded surface $\Sigma$ 
in $\mathbb{R}^3\times [0,1]$ with $\partial\Sigma \subset \mathbb{R}^3\times\{0,1\}$ and 
$\partial\Sigma \cap (\mathbb{R}^3\times\{i\})= L_i$ for $i=0,1$. Given a smooth, 
compact cobordism $\Sigma$ in $S^3\times [0,1]$ from $L$ to $L'$, a \textit{movie} of the link cobordism $\Sigma$ is a finite 
collection of link diagrams $\{D_{i}\}_{i=0}^{k}$ for a non-negative integer $k$, such that the link diagram $D_{i}$ 
is related to $D_{i+1}$ by either a single birth, a single saddle, a single death, or a Reidemeister move, localized 
to a disk in $\mathbb{R}^2$. 	
By work of Carter and Saito \cite{CS93}, every smooth cobordism $\Sigma$ as above has a movie. Additionally, they proved, 
if two smooth cobordisms $\Sigma,\Sigma'\subset \mathbb{R}^3\times[0,1]$ between links $L_0$ and $L_1$ are ambient isotopic 
relative to the boundary, then their movies differ by finitely many \emph{movie moves} (see also \cite{Bar05}, and \cite{Kho06}). Let us fix our convention and by isotopy, we refer to ambient isotopy relative to the boundaries.

Movies associated to isotopic link cobordisms are related by a sequence of \emph{movie moves}, which locally adjust the frames of a movie \cite{CS93}. For a list of movie moves, see figures 5-9 in \cite{Kho06}. For a more detailed treatment of the movie moves, we refer the reader to \cite{Bar05,Kho06,LLS21}. 

\begin{Definition}\label{Def:EquiCob}
 For periodic links $L_0$ and $L_1$ in $\mathbb{R}^3\backslash\boldsymbol{z}$, a smooth cobordism 
 $\Sigma\subset (\mathbb{R}^3\backslash\boldsymbol{z})\times[0,1]$ is called an \emph{equivariant cobordism} if it is 
 invariant under the extended action $\tilde{\theta}$, 
 and is disjoint from the annulus $\tilde{U}$.
\end{Definition}

We show that equivariant cobordisms $\Sigma\subset (\mathbb{R}^3\backslash\boldsymbol{z})\times [0,1]$ 
have \textit{equivariant movie} presentations. Two equivariant cobordisms $\Sigma$ and $\Sigma'$ are called \emph{equivariantly isotopic}, 
if there is an ambient isotopy of $(\mathbb{R}^3\backslash\boldsymbol{z})\times[0,1]$ relative to boundary such that the isotopy commutes 
with the action $\tilde{\theta}$. For a smooth equivariant cobordism, we define, 
\begin{Definition}\label{Def:EquiMM}
	An \emph{equivariant movie move} means transforming the movie of a $p$-equivariant cobordism by  Carter Saito movie moves localized 
	to $p$ disjoint disks in $\mathbb{R}^2\backslash\boldsymbol{0}$ such that the disks and the movies correspond under the action of 
	$\mathbb{Z}_p$ by rotation on $\mathbb{R}^2\backslash\boldsymbol{0}$.
\end{Definition}
Analogous to the result above, we show 
\begin{IntroTheorem}\label{Thm:EquiMM}
	Fix two equivariantly isotopic, equivariant cobordisms $\Sigma$ and $\Sigma'$ from $L_0$ to $L_1$ represented by equivariant movies. 
	Then the equivariant movie of $\Sigma$ differs from the equivariant movie of $\Sigma'$ by finitely many equivariant movie moves.
\end{IntroTheorem}

The \emph{Khovanov homology} is a link invariant that assigns to the link diagram of $L$ a bigraded chain complex of 
$R$-modules $C_{\text{Kh}}(L;R)$, for any unital commutative ring $R$. The graded Euler characteristic of Khovanov homology is the 
Jones polynomial. Moreover, Khovanov homology assigns a chain map to any elementary string move (birth, saddle, death, or Reidemeister move) \cite{Bar05,Kho00,Kho02}.

Assume that $\Sigma$ and $\Sigma'$ are smooth, ambient isotopic cobordisms in $\mathbb{R}^3\times [0,1]$ with boundary 
$\partial\Sigma=\partial\Sigma'=L_0\coprod L_1$. Jacobson \cite{Jac04}, Bar-Natan \cite{Bar05}, and Khovanov \cite{Kho06} 
proved that the maps induced on 
Khovanov homology (denoted by $\text{Kh}(L)$ for a link $L$) by $\Sigma$ and $\Sigma'$ are equal up to a sign i.e., 
$\text{Kh}(\Sigma)=\pm\text{Kh}(\Sigma')$. This property of Khovanov homology is called \emph{functoriality}.

A smooth \emph{concordance} between knots $K_0$ and $K_1\subset S^3$ is a smooth embedded annulus 
$A:S^1\times [0,1]\hookrightarrow S^3\times [0,1]$ with $K_1=A(S^1\times\{1\})\subset S^3\times \{1\}$ and 
$K_0=A(S^1\times\{0\})\subset S^3\times \{0\}$. We also call the 
image of the annulus $A(S^1\times [0,1])=F$ a concordance between $K_1=A(S^1\times \{1\})$ and $K_0=A(S^1\times \{0\})$. 
If the projection $S^3\times [0,1] \rightarrow [0,1]$ is a Morse function when restricted to $F$ with only critical points of
index 0 and 1 (so no local maxima), then we say that $F$ is a \emph{ribbon concordance} between $K_0$ and
$K_0$.

As a consequence of the functoriality of Khovanov homology, one can study Khovanov homology of a ribbon concordance. 
A framework for studying ribbon concordances is Bar-Natan's dotted cobordism treatment of Khovanov homology \cite{Bar05}. 
Levine and Zemki \cite{LZ19} 
showed that if $F$ is a ribbon concordance from $L_0$ to $L_1$, then the induced map 
$\text{Kh}(F):\text{Kh}(L_0)\rightarrow \text{Kh}(L_1)$ is a bigraded split injection. 

Given a finite group $G$ and a $G$-space $X$, \textit{equivariant (co)homology} (also known in the literature as the 
Borel equivariant homology or Cartans mixing construction) is effective for studying the (co)homological behavior of both the space and 
the $G$-action \cite{WyHsiang75}. The equivariant (co)homology of $X$ is given by the (co)homology $H^G_*(X):=H_*(E\times_GX)$ 
(respectively $H_G^*(X):=H^*(E\times_GX)$) where $EG$ is the universal $G$-bundle.

Fix $R$ to be ring, and assume we have an $R$-chain complex $(C,d)$ with a $G$-action, the equivariant (co)homology of $C$, computes the group (co)homology of an $R[G]$-module $M$, with coefficients in the $R[G]$-chain complex $(C,d)$ \cite{KSBrown82}. 
Hence, the \textit{equivariant Khovanov homology} of a $p$-periodic link for a choice of unital ring $R$ is defined as follows:
\begin{Definition}
	The \emph{equivariant Khovanov homology} of the $p$-periodic link $L$ is defined as
	\begin{equation}
		\mathrm{EKh}(L,M)=\mathrm{Ext}_{R[\mathbb{Z}_p]}(M,C_{\mathrm{Kh}}(L;R))
	\end{equation}
where $M$ is a $R[\mathbb{Z}_p]$-module, and $C_{\mathrm{Kh}}(L;R)$ is the Khovanov chain complex.
\end{Definition}
Equivariant Khovanov homology for $p$-periodic links was recently studied by Borodzik and Politarczyk \cite{BP17}. 
They also studied equivariant Bar-Natan and Lee homologies for periodic knots. We restate the construction of 
equivariant Khovanov homology in section \ref{Chp:EquKhovInv}.

In section \ref{Chp:Result}, we show that equivariant Khovanov homology is functorial up to a sign. More specifically, 
we prove that the chain maps induced on the equivariant Khovanov chain complex by equivariantly isotopic 
equivariant cobordisms are chain homotopic up to a factor of $(\pm 1)^p$.

\begin{IntroTheorem}\label{Thm:EquiFunct}
	Given two equivariantly isotopic equivariant link cobordisms $\Sigma,\Sigma'$ from a $p$-periodic link $L_0$ to $L_1$, 
	the induced morphism 
    \begin{equation}
    \mathrm{EKh}(\Sigma),\mathrm{EKh}(\Sigma'):\mathrm{EKh}(L_0;\mathbb{Z})\rightarrow\mathrm{EKh}(L_1;\mathbb{Z})    
    \end{equation}
    differ by a factor of $(\pm1)^p$.
\end{IntroTheorem}

Lastly, in the equivariant case, we show that there is an equivariant analogue of the neck cutting relation that holds in 
equivariant Khovanov homology (see section \ref{SSec:EquiNecCut}). Therefore, we conclude that,

\begin{IntroTheorem}\label{Thm:EquiRibbCon}
    Fix a smooth equivariant ribbon concordance $F$ between periodic knots $K$ and $K'$. The map induced by $F$ on 
	equivariant Khovanov homology is a split injection.
\end{IntroTheorem}

\paragraph{Organization:} In section \ref{Sec:Background}, we study the periodic links, the Khovanov invariant for tangles, 
and the results about functoriality of Khovanov homology. Moreover, we restate the construction of equivariant Khovanov homology for periodic links. In section \ref{Chp:Result}, we prove the functoriality of the equivariant Khovanov homology. Lastly, in section \ref{Sec:EquConcObs} we show that the equivariant Khovanov homology obstructs to equivariant ribbon concordance.

\paragraph{Acknowledgment:} Author would like to thank Robert Lipshitz for his constant support through the completion of 
this manuscript. Also, the Author thanks Taylor Lawson for his helpful assist on the proof of the theorem \ref{Thm:EquiFunct}.

\section{Background}\label{Sec:Background}
In this section we will introduce some background on Khovanov invariant for tangles that we will use throughout the arguments in this paper. We also restate the construction of equivariant Khovanov homology for periodic links. Lastly, we review the skew group algebras and their modules.

\subsection{Tangles and diskular tangles}\label{SSec:DiskTang}
Our approach to the tangle decomposition of a link is the frameworks introduced in \cite{Bar05,LLS21}. We utilize the \textit{planar arc diagram} introduced in \cite[\S5]{Bar05}. This coincides with \textit{diskular tangles} in \cite{LLS21}. We begin by summarizing some of the definitions and theorems of \cite{LLS21}.

Let $\mathbb{D}$ be the standard closed disk centered at the origin in $\mathbb{C}$ and $\mathbb{D}_i\subseteq \mathbb{D}$ be disjoint 
open sub-disks for $i=1,...,k$ such that $\overline{\mathbb{D}}_i\subset \overset{\circ}{\mathbb{D}}$. A $k$-\textit{punctured disk} $\mathcal{D}$ 
is
\begin{equation}
	\mathcal{D}=\mathbb{D}\backslash \big({\mathbb{D}}_1\cup{\mathbb{D}}_2\cup\cdots\cup{\mathbb{D}}_k\big)
\end{equation}
We can partition the boundary of $\mathcal{D}$ in to two sets: The outer boundary $\partial \mathbb{D}=S^1$ and $k$ inner boundaries 
$\partial_i\mathcal{D}=\partial\mathbb{D}_i$ for $i=1,...,k$. Up to scaling and translating, any one of those inner boundaries is the unit 
circle $S^1=\{z\in\mathbb{C}\,\vert \,\lvert z\rvert=1\}$. Hence, we treat all boundary components as the unit circle. For an arbitrary 
non-negative integer $m$, a set of $m$-\textit{marked points} on a unit circle consists of $m$ points 
$\{p_j\,\vert\,p_j\neq 1, j=1,...,m\}\subset \partial\mathbb{D}$. For example, we can choose 
$\displaystyle{\{p_{1}=e^{\frac{2\pi i}{m+1}},...,\;p_{m}=e^{\frac{2\pi i m}{m+1}}\}}$ i.e., non-trivial roots of unity as our 
$m$-marked points. 

For a choice of boundary components of $\mathcal{D}$ and any 
positive integer $m$, we can set $m$-marked points on that boundary component. The importance of such choices will be more clear in the gluing process.

\begin{Definition}\label{Def:DiskTang}
	An $(m_1,...,m_k;n)$-\textit{diskular tangle} $T=T^{(m_1,...,m_k;n)}
$ is a diagram of a tangle in a thickened $k$-punctured disk $\mathcal{D}$ with the following information:
	\begin{enumerate}[label=(DT\arabic*)]
		\item The $n,m_1,...,m_k$ are fixed, positive, even integers.
		\item There are $m_i$ marked points on $\partial_i\mathcal{D}$ ($i$th inner boundary of $\mathcal{D}$) and $n$ marked points on the outer 
		boundary of $\mathcal{D}$.
		\item The tangle diagram $T$ consist of finitely many immersed circles and $\frac{1}{2}(n+m_1+...+m_k)$ arcs in $\mathcal{D}$ with the 
		boundary of the arcs on the marked points. 
		\item The arcs are perpendicular (making the right angle with the tangent vectors to $\partial \mathcal{D}$ at marked points) near the 
		boundary of $\mathcal{D}$. 
	\end{enumerate}
\end{Definition}

Figure \ref{Disktangcompo} (a) and (b) depicts two diskular tangles. If there are no inner boundaries, and $n$-marked points on the outer boundary, we denote the diskular tangle by $T^{(;n)}$. By contrast, $T^{(0,...,0;m)}$ would imply that there are inner boundaries with no marked points, and one outer boundary with $m$-marked points.

Two diskular tangles $S=S^{(l_1,...,l_j;m_i)}$ and $R=R^{(m_1,...,m_i,...,m_k;n)}$ can be composed by gluing $S$ from its outer boundary to 
the $i$th inner boundary of $R$ to form a\break $(m_1,...,m_{i-1},l_1,...,l_j,m_{i+1},...,m_k;n)$-diskular tangle denoted by $R\circ_i S$. 
The reader should note that here the subscript $\circ_i$ specifies that we glue $S$ into $R$'s $i$th boundary. In the case of unoriented tangles, 
to be able to get a diskular tangle from $R\circ_i S$ we need the number of marked points on the $i$th inner boundary of $R$, $m_i$, to be equal to the number of marked points on the outer boundary of $S$. However, if one works with oriented tangles, then the orientation of arcs in $R$ that 
meet the $\partial_iR$ must match with the orientation of arcs that start or terminate on the outer boundary of $S$. Gluing can be made canonical 
if we identify $\partial_iR$ with $ \partial_{\text{out}}S$ by scaling and translating of $S$ and $R$ (Figure \ref{Disktangcompo} (c)).

\begin{figure}[ht]
\centering
\begin{subfigure}[c]{0.3\textwidth}
	\def\svgwidth{1\textwidth}
	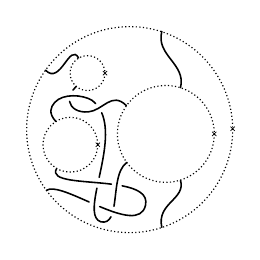

    \vspace{-0.5cm}\caption{}
\end{subfigure}
\begin{subfigure}[c]{0.3\textwidth}
	\def\svgwidth{1\textwidth}
	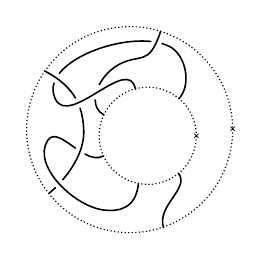

    \vspace{-0.5cm}\caption{}
\end{subfigure}
\begin{subfigure}[c]{0.3\textwidth}
	\def\svgwidth{1\textwidth}
	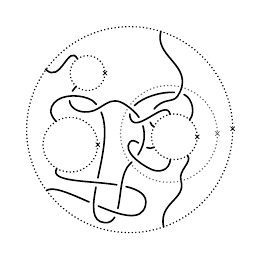

    \vspace{-0.5cm}\caption{}
\end{subfigure}
\caption[Gluing diskular tangles]{\small (a) Shows a diskular tangle $T^{(2,2,4;4)}$ and (b) shows a diskular tangle $T^{(6;4)}$. (c) depicts 
the composition $T^{(2,2,4;4)}\circ_3 T^{(6;4)}$ which is a $(2,2,6;4)$-diskular tangle.}
\label{Disktangcompo}
\end{figure}

Similarly, the $(m_1,\cdots,m_k;n)$-\textit{flat} diskular tangles are defined similarly to definition \ref{Def:DiskTang} but we require a 
smooth and proper embedding of $\frac{1}{2}(n+m_1+\cdots+m_k)$ arcs and finitely many circles in a $k$-punctured disk $\mathcal{D}$. 
We denote by $B^{(m_1,\cdots,m_k;n)}$ the collection of $(m_1,...,m_k;n)$-flat diskular tangles. 

\begin{Definition}\label{Def:FlaDisTang}
    The $(m_1,\cdots,m_k;n)$-\textit{flat diskular tangles} are defined similarly to definition \ref{Def:DiskTang} (see \cite[section 4.2.2]{LLS21}) but we require a smooth and proper embedding of $\frac{1}{2}(n+m_1+\cdots+m_k)$ arcs and finitely many circles in a $k$-punctured disk $\mathcal{D}$. We denote by $B^{(m_1,\cdots,m_k;n)}$ the collection of $(m_1,...,m_k;n)$-flat diskular tangles.
\end{Definition}

The collection of $(m;n)$-flat diskular tangles in the annulus $A\subset \mathbb{C}$ with no closed components and arcs that avoid the ray 
$r_1=\{(0,t)\;\vert\, t\in [0,+\infty)\}\cap A$ up to ambient isotopy of $\mathcal{D}$ relative to the boundary is denoted by $\boldsymbol{B}^{(m;n)}$. We use $1\in S^1\subset \mathbb{D}$ as the base point. 
\begin{Remark}\label{Rmk:BasePoint1}
    The choice of $1\in S^1$ as the base point on the boundaries of $\mathcal{D}$ might appear excessive to the reader as we only allow tangle 
	diagrams to be scaled or translated. This choice of base point is indeed extra here. However, it will be used in section \ref{SSec:KhovanovInv}.
\end{Remark}

In \cite{Kho02} tangles has been considered in $\mathbb{R}\times[0,1]$ with boundary on $\mathbb{N}\times\{0,1\}$. After scaling one can consider tangle diagrams in $[0,1]\times[0,1]$ with boundary of the arcs on $\{\frac{1}{n+1},\cdots\frac{n}{n+1}\}\times\{0\}$ and $\{\frac{1}{m+1},\cdots\frac{m}{m+1}\}\times\{1\}$. Identifying the right and left side of the square $[0,1]^2$ will transform the square into an annulus. 
This also provides a one-to-one correspondence between tangle diagrams in $[0,1]^2$ and diskular tangles ${T}^{(m;n)}$ in the annulus 
$A=S^1\times [0,1]$ that do not intersect the ray $\{1\}\times[0,1]$. Given an $(m;n)$-flat diskular tangle $R\in B^{(m;n)}$ we can reflect 
the arcs radially around the middle circle in the annulus and denote the resulting flat diskular tangle  $\widehat{R}\in B^{(m;n)}$ 
(figure \ref{Fig:FlatDiskTang}). 

\begin{figure}[ht]
	\centering
    \begin{subfigure}[c]{0.25\textwidth}
	\def\svgwidth{0.9\textwidth}
	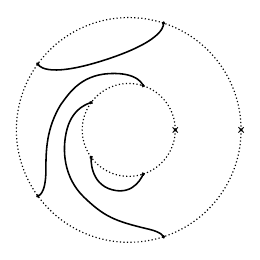

        \caption{}
    \end{subfigure}
    \begin{subfigure}[c]{0.25\textwidth}
	\def\svgwidth{0.9\textwidth}
	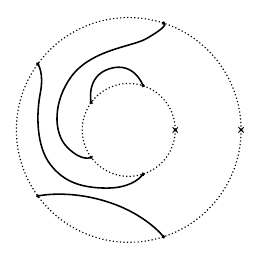

        \caption{}
    \end{subfigure}
	\caption[Example of flat diskular tangles]{\small (a) Shows a flat ${(4;4)}$-diskular tangle and (b) shows the reflection of same flat 
	diskular tangle ${\widehat{T}^{(4;4)}}$.}
	\label{Fig:FlatDiskTang}
\end{figure}

\begin{Definition}\label{Def:SimpCob}
	Given two $(m_1,...,m_k;n)$-diskular tangle $S,T$, by an \emph{elementary cobordism} ${\mathfrak{s}: S\rightarrow T}$, we mean any of following:
\begin{itemize}
	\item A  planar isotopy $f_t:\mathcal{D}\rightarrow \mathcal{D}$ of $k$-punctured disks such that $f_t$ restricted to the boundary is 
	the identity for all $t$.
	\item Any of the Reidemeister moves away from the boundary.
	\item Any of the Morse moves (birth, death, and adding a saddle). 
\end{itemize}
\end{Definition}
%Figure \ref{Fig:SadCob} shows a composition of two saddle addition elementary cobordisms. 

In \cite[ Lemma 4.7]{LLS21}, the authors showed that collection of diskular tangles forms a multicategory $\mathbb{T}$ enriched in categories. 
In this multicategory, an unoriented $(m_1,...,m_k;n)$-diskular tangle $S=S^{(m_1,\cdots,m_k;n)}$ is a multi-morphism between objects 
$m_1, m_2,\cdots, m_k$ and $n$ i.e., $S\in \text{Hom}_{\mathbb{T}}(m_1,\cdots, m_k;n)$. The identity morphism is the tangle 
$\text{Rad}_n\in \text{Hom}_{\mathbb{T}}(n;n)$ which is the radial crossingless matching of $n$-marked points on the outer boundary and the 
$n$-marked points on the inner boundary. In this setting, any elementary cobordism (definition \ref{Def:SimpCob}) between 
$(m_1,...,m_k;n)$-diskular tangles $S,T\in \text{Hom}_{\mathbb{T}}(m_1,...,m_k;n)$, is a 2-morphisms in $\mathbb{T}$.

In light of the categorical refinement of diskular tangles in \cite[Section 4]{LLS21}, we denote the collection of diskular tangles for 
fixed integers $m_1,...,m_k,n$ by $\mathbb{T}^{(m_1,\cdots,m_k;n)}$. For instance, instead of writing the $S=S^{(m_1,\cdots,m_k;n)}$ 
we will write $S\in \mathbb{T}^{(m_1,\cdots,m_k;n)}$ to denote an $(m_1,\cdots,m_k;n)$-diskular tangle.

\begin{figure}[ht]
	\centering
	\def\svgwidth{0.7\textwidth}
	%% Creator: Inkscape 1.2.2 (b0a8486, 2022-12-01), www.inkscape.org
%% Accompanies image file 'Saddle1.pdf' (pdf, eps, ps)
\begingroup%
  \makeatletter%
  \providecommand\color[2][]{%
    \errmessage{(Inkscape) Color is used for the text in Inkscape, but the package 'color.sty' is not loaded}%
    \renewcommand\color[2][]{}%
  }%
  \providecommand\transparent[1]{%
    \errmessage{(Inkscape) Transparency is used (non-zero) for the text in Inkscape, but the package 'transparent.sty' is not loaded}%
    \renewcommand\transparent[1]{}%
  }%
  \providecommand\rotatebox[2]{#2}%
  \newcommand*\fsize{\dimexpr\f@size pt\relax}%
  \newcommand*\lineheight[1]{\fontsize{\fsize}{#1\fsize}\selectfont}%
  \ifx\svgwidth\undefined%
    \setlength{\unitlength}{309.60001373bp}%
    \ifx\svgscale\undefined%
      \relax%
    \else%
      \setlength{\unitlength}{\unitlength * \real{\svgscale}}%
    \fi%
  \else%
    \setlength{\unitlength}{\svgwidth}%
  \fi%
  \global\let\svgwidth\undefined%
  \global\let\svgscale\undefined%
  \makeatother%
  \begin{picture}(1,0.31395351)%
    \lineheight{1}%
    \setlength\tabcolsep{0pt}%
    \put(0,0){\includegraphics[width=\unitlength,page=1]{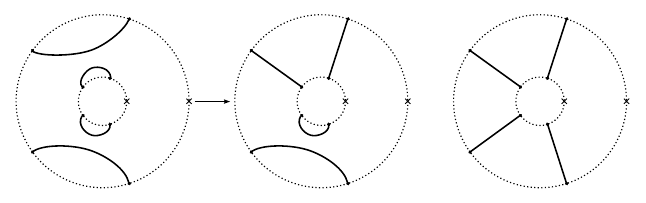}}%
    \put(0.31371123,0.16860468){\color[rgb]{0,0,0}\makebox(0,0)[lt]{\lineheight{1.25}\smash{\begin{tabular}[t]{l}$\mathfrak{s}$\end{tabular}}}}%
    \put(0,0){\includegraphics[width=\unitlength,page=2]{Saddle1.pdf}}%
    \put(0.65285851,0.16860468){\color[rgb]{0,0,0}\makebox(0,0)[lt]{\lineheight{1.25}\smash{\begin{tabular}[t]{l}$\mathfrak{s}$\end{tabular}}}}%
  \end{picture}%
\endgroup%

    \vspace{-0.5cm}\caption[Example of a sequence of elementary cobordisms]{\footnotesize Composition of two Morse saddle elementary cobordism between 
	the $(4;4)$-flat diskular tangle $R$ on the left, and $(4;4)$-diskular tangle $\text{Rad}_4$ on the right.}
    \label{Fig:SadCob}
\end{figure} 

\subsection{Periodic links and tangles}\label{SSec:PeriodKnots}
In this section we study the $p$-periodic links in more depth and introduce periodic tangle decompositions 
for $p$-periodic links. 

For simplicity, we use the term periodic links instead of $p$-periodic link when it will not cause confusion.

Periodic links have periodic diagrams. That is to say, there is a generic position of the periodic link $L$ such that the image of $L$ by 
the projection $\mathbb{R}^3\backslash \boldsymbol{z}\rightarrow\mathbb{R}^2\backslash\boldsymbol{0}$ denoted by $D\subset \mathbb{R}^2$, 
is a link diagram that misses the origin, and is taken to itself by $\frac{2\pi}{p}$ rotation  around the origin. Analogous to ordinary 
links, one can define the equivariant Reidemeister moves as transformations on a diagram of a periodic link.
\begin{Definition}
 A \emph{$p$-equivariant Reidemeister move} means applying a Reidemeister move to a periodic link diagram of a periodic link, localized on 
 $p$ disjoint, closed disks in $\mathbb{R}^2\backslash\boldsymbol{0}$, so that disks and Reidemeister moves correspond under the action $\mathbb{Z}_p$. 
\end{Definition}
Figure \ref{fig:EquRM} depicts an example of a 5-equivariant Reidemeister move. Given a $p$-periodic link $L$ in $\mathbb{R}^3\backslash\boldsymbol{z}$ 
(similarly in $S^3\backslash U$, where $U$ is the fixed unknot), the quotient link $\overline{L}=L/\mathbb{Z}_p$ is a link in 
$\mathbb{R}^3\backslash \boldsymbol{z}$ (respectively in $S^3\backslash U$). Throughout this paper, an overline indicates the quotient of a periodic 
link by $\mathbb{Z}_p$. Both $L$ and $\overline{L}$ can be considered as annular links by considering them as links in $\mathbb{R}^3\backslash \boldsymbol{z}$. 
The projection to the $xy$-plane will provide an annular link diagram of $L$ and $\overline{L}$ in an annulus centered at origin of $\mathbb{R}^2$. 

\begin{figure}[ht]
	\centering
    \begin{subfigure}[c]{0.4\textwidth}
	\def\svgwidth{1\textwidth}
	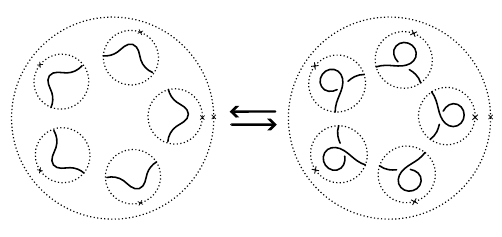

        \vspace{-0.5cm}\caption{}
    \end{subfigure}
    \begin{subfigure}[c]{0.4\textwidth}
	\def\svgwidth{1\textwidth}
	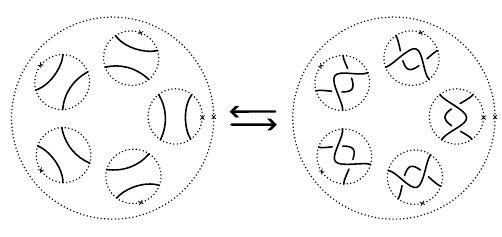

        \vspace{-0.5cm}\caption{}
    \end{subfigure}
	\caption[Equivariant Reidemeister move]{(a) Equivariant type I Reidmeister move and (b) Equivariant type II Reidmeister moves for a $5$-periodic link.}
	\label{fig:EquRM}
\end{figure}
\begin{Theorem}
Two periodic links $L_0$ and $L_1$ represented by periodic link diagrams $D_0$ and $D_1$ are \textit{equivariantly isotopic} if and only if $D_0$ is obtained 
from $D_1$ by a finite sequence of equivariant Reidemeister moves. 
\end{Theorem}
\begin{proof}
	Consider the quotient link diagram $\overline{D}_0$. Any annular Reidemeister move on $\overline{D_0}$ can be pulled back to an equivariant Reidemeister 
    move. Therefore, the theorem follows from the similar result in annular links topology \cite[Theorem 1]{HP89}. 
\end{proof}

The action of $\mathbb{Z}_p$ on $\mathbb{R}^3$ (respectively $S^3$) denoted by $\theta$, can be extended to an action on $\mathbb{R}^3\times[0,1]$ (respectively $S^3\times [0,1]$) if we let the action be trivial on $[0,1]$. This extended action has fixed annulus $U\times [0,1]$. We will use this extended action in section \ref{SSec:EquiFunct}. We also introducing a notion of admissible periodic diskular tangles which we will use in the proof of theorem \ref{Thm:EquiFunct}. 

\begin{Definition}\label{Def:PerDisTan}
   An $(m_1,...,m_k;n)$-diskular tangle $E$ is an \emph{admissible} $p$-\textit{periodic diskular tangle}, if the followings hold.
   \begin{enumerate}[label=(PD\arabic*)]
    \item Each inner boundary component has a distingueshed point $p_0\in \partial_i\mathcal{D}$ as the base point. Also, the base points on inner boundaries correspond under the action of $\mathbb{Z}_p$. 
    \item We have $n=0$, i.e., there are no arcs in the diagram of $E$ with boundary on $\partial_{out }\mathcal{D}$, the outer boundary of $\mathcal{D}$ for a $k$ punctured disk $\mathcal{D}$ centered at origin in $\mathbb{R}^2$.
    \item\label{Itm:equal} We have $k=p$ and the $\mathbb{Z}_p$ acts on the set of inner boundaries $\partial_{in}\mathcal{D}=\{\partial_i\mathcal{D}\}_{i\in \mathbb{Z}_p}$ by cyclic permutation, i.e., $\theta(\partial_i\mathcal{D})=\partial_{i+1}\mathcal{D}$ for $i\in \mathbb{Z}_p$.
    \item The arcs and circles in $E$ correspond under the rotation by $\mathbb{Z}_p$.
   \end{enumerate}
\end{Definition}

\begin{figure}[ht]
	\centering
    \begin{subfigure}[l]{0.3\textwidth}
	\def\svgwidth{1\textwidth}
	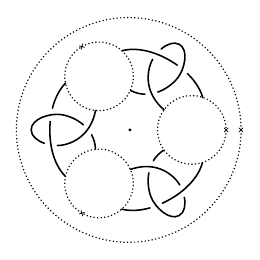

        \vspace{-0.5cm}
        \caption{}
    \end{subfigure}
    \begin{subfigure}[r]{0.3\textwidth}
	\def\svgwidth{1\textwidth}
	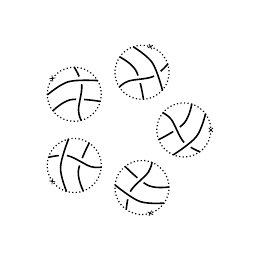

        \vspace{-0.5cm}
        \caption{}
    \end{subfigure}
	%\subfloat[]{\incfig{0.45}{EFDTang1a}}
	%\hspace{2cm}
	%\subfloat[]{\incfig{0.45}{EFDTang2}}
	%\subfloat[]{\incfig{0.35}{DTang0-4}}
	%\hspace{-0.5cm}
	\caption[Example of periodic diskular tangles.]{\small (a) Shows an admissible 3-periodic diskular tangle $T^{(4,4,4;)}$. (b) Shows a 5-equivariant diskular tangle $(T_0,...,T_4)$ with $T_0,=\cdots=T_4\in\mathbb{T}^{(;6)}$.}
	\label{Fig:EquiFlatDiskTang}
\end{figure}

A few notes about the definition \ref{Def:PerDisTan}: 
\begin{enumerate}
    \item In \ref{Itm:equal}, the requirement $k=p$ could be generalized to $k$ being divisible by $p$. Then the inner boundaries 
    $\partial_{in}\mathcal{D=}\{\partial_i\mathcal{D}\}_{i=1}^k$ can be partitioned into subsets $\sqcup_{l=1}^p\mathcal{I}_l$ 
    that correspond by cyclic permutation by action of $\mathbb{Z}_p$. However, this generalization is not necessary for the proof of the main results in this paper. Moreover, \ref{Itm:equal} implies that the number of marked points on inner boundaries are equal, i.e., $m_{1}=...=m_{k}$. Moreover, the action of $\mathbb{Z}_p$ maps the marked points on $\partial_i\mathcal{D}$ bijectively to marked points on $\partial_{i+1}\mathcal{D}$.
    \item In contrast to definition \ref{Def:DiskTang}, for each $i=1,...,k$ there are $m_i+1$ points on the inner boundary $\partial_i\mathcal{D}$. One of the $m_i+1$ points is set to be the base point.
\end{enumerate}

\begin{Definition}\label{Def:EquiDiskTang}
    By a $p$-\emph{equivariant diskular tangle}, we refer to $p$ disjoint copies of a diskular tangles 
    $T_0,...,T_{p-1}\in \mathbb{T}^{(m_1,...,m_k;n)}$ in $\mathbb{R}^2\backslash\boldsymbol{0}$ such that
    \begin{itemize}
        \item Each boundary component has a distingueshed base point $p_0\in \partial_i\mathcal{D}$.
        \item We have the action $\mathbb{Z}_p$ on $\{T_j\}_{j\in \mathbb{Z}_p}$ with $\theta(T_j)=T_{j+1}$. 
    \end{itemize}
\end{Definition}

We can extended the notion of elementary cobordims (definition \ref{Def:SimpCob}) to periodic diskular tangles. 
\begin{Definition}\label{Def:EquEleCob}
    A $p$-\emph{equivariant elementary cobordism} $\tilde{\mathfrak{s}}$ between $p$-periodic diskular tangles $E$ and $E'$, also denoted by 
    $\tilde{s}:E\rightarrow E'$, is applying an elementary cobordism to $E$ localized to $p$ disjoint disks away from the center of $E$ such that 
    the elementary cobordism correspond under the under the rotation by $\mathbb{Z}_p$.
\end{Definition}

\subsection{Khovanov invariant for tangles}\label{SSec:KhovanovInv}
In this section, we review the construction of the Khovanov invariant for diskular tangles (section \ref{SSec:DiskTang}).

For a unital ring $\mathbb{F}$, let $\mathcal{A}=\mathbb{F}[\boldsymbol{X}]/(\boldsymbol{X}^2)$ which is a Frobenius algebra by defining a comultiplication and counit \cite{Bar05,Kho00,Kho02}. Also, $\mathcal{A}$ can be considered as a graded $\mathbb{F}$-module by choosing $\text{gr}_q(\boldsymbol{1})=-1$ and $\text{gr}_q(\boldsymbol{X})=1$. For instance, the choice $\text{gr}_q(\boldsymbol{1})=0$ and $\text{gr}_q(\boldsymbol{X})=2$ makes $\mathcal{A}$ a graded $\mathbb{F}$-algebra. In the literature, $\text{gr}_q$ is referred to as the \textit{quantum grading} \cite{Kho00,Bar05}. 

To construct the Khovanov chain complex we use the \emph{Khovanov TQFT} denoted by $\mathcal{F}:\mathcal{C}ob^{1+1}\rightarrow \mathbb{F}-\text{Mod}$ (see \cite{Kho00}) induced by $\mathcal{A}$. Here, $\mathcal{C}ob^{1+1}$ is the category consisting of closed 1-dimensional 
manifolds as objects, and surfaces with boundary as morphisms.

\begin{Definition}\label{Def:ArcAlg}
    Let $\mathbb{F}$ be as above. For an even positive integer $n$, we define the \emph{arc algebra} $\mathcal{H}^n$ as follows. 
	\begin{equation}\label{Equ:ArcAlg}
			\mathcal{H}^n=\bigoplus_{a,b\in \boldsymbol{B}^{(0;n)}}\mathcal{F}(\hat{a}\circ b)\{n/2\}
	\end{equation}
where $\boldsymbol{B}^{(0;n)}$ denotes the collection of isotopy classes of flat $(0;n)$-diskular tangles with no closed components (section \ref{SSec:DiskTang}), and $\mathcal{F}$ is the Khovanov TQFT functor.
\end{Definition}

We use a convention similar to \cite{Bar05} for grading shift operations. For a bigraded $\mathbb{F}$-module 
${\displaystyle A=\bigoplus_{(i,j)\in \mathbb{Z}\oplus\mathbb{Z}} A^{i,j}}$, we denote by $[n]$ and $\{m\}$ the 
following shifts in the bigrading.
\begin{equation}
        (A\{m\}[n])^{i,j} = A^{i-n,j-m}. 
\end{equation}

\begin{Remark}\label{Rmk:BasePoint}
    The condition that arcs in $\boldsymbol{B}^{(0;n)}$ do not intersect the ray $r_1=\{1\}\times[0,1]$, is used to ensure that the arc algebra $\mathcal{H}^n$ 
    defined above is identical to the ring $H^n$ defined in \cite[section 2.4]{Kho02}.
\end{Remark}
Now we can introduce the construction of Khovanov invariant for diskular tangles.
\begin{Lemma}\label{Lem:ArcAlgRing}\cite{LLS21}
    The $\mathbb{F}$-module $\mathcal{H}^n$ is an associative ring.     
\end{Lemma}

First, to a flat $(m_1,...,m_k;n)$-diskular tangle $R$, we can assign a graded $\mathbb{F}$-module 
\begin{equation}\label{Equ:FlatModule}
	V(R)=\bigoplus_{a_i\in \boldsymbol{B}^{(0;m_i)},b\in\boldsymbol{B}^{(0;n)}} \mathcal{F}(\hat{b}\circ R\circ(a_1,...,a_k))\{\frac{n}{2}\}.
\end{equation}
where $\{\cdot\}$ denotes the shift in quantum grading $\mathrm{gr}_{q}$.

The $V(R)$ defined above, is an $(\mathcal{H}^{m_1}\otimes_{\mathbb{F}}\cdots \otimes_{\mathbb{F}}\mathcal{H}^{m_k},\mathcal{H}^{n})$-bimodule \cite[section 4]{LLS21}.

Fix an $(m_1,...,m_k;n)$-diskular tangle $T$ with $N$ crossings. For each $\alpha\in \underline{2}^N$, let $T^{\alpha}$ denote the $\alpha$-resolution of $T$ according to 
\begin{equation}\label{Equ:Resolution}
	\begin{tikzcd}[scale=0.2]
		\phantom{A} \arrow[r, bend right=55, no head] & \phantom{A}&\phantom{A}\arrow[rd, dash, ]   \arrow[l, shift left=15ex, "1"']& 
		\phantom{A}\arrow[ld, crossing over, dash]   \arrow[r,  shift left=-15ex, "0"]& \phantom{A}\arrow[d, bend left=70, no head, shift left=1ex] & 
		\phantom{A}\arrow[d, bend right=70,  no tail, no head] \\
		\phantom{A}\arrow[r, bend left=60, shift left=-1.5ex, no head] & \phantom{A}& \phantom{A}& \phantom{A}& \phantom{A}&\phantom{A}
	\end{tikzcd} 
\end{equation}
Therefore, $T^{\alpha}$ is  a flat $(m_1,...,m_k;n)$-diskular tangle and by (\ref{Equ:FlatModule}) we can assign a $(\mathcal{H}^{m_1}\otimes_{\mathbb{F}}\cdots \otimes_{\mathbb{F}}\mathcal{H}^{m_k},\mathcal{H}^{m})$-bimodule $V(T^{\alpha})$ to $T^{\alpha}$.

To build a tangle invariant for an $(m_1,...,m_k;n)$-diskular tangle $R$ with $N$ crossings, we assign a chain complex $C(R)=(C^{i}(R),d^i)$ of  $(\mathcal{H}^{m_1}\otimes_{\mathbb{F}}\cdots \otimes_{\mathbb{F}}\mathcal{H}^{m_k},\mathcal{H}^{m})$-bimodules to $R$ as follows
\begin{equation}
	C^{i}(R)=\bigoplus_{\substack{|\alpha|=i\\ \alpha\in \underline{2}^N}}\big(V(R^{\alpha})\big)\{|\alpha|+N_+-2N_-\}
\end{equation}
where $R^{\alpha}$ denotes the resolution of $R$ for $\alpha\in\underline{2}^{N}$, $\mathcal{F}$ is the Khovanov TQFT functor defined above, and $N_+$ and $N_-$ denote the number of positive and negative crossings respectively. We refer the reader to \cite{Kho00,Bar05} for the definition of the differentials $d^i:C^{i}(R)\rightarrow C^{i+1}(R)$. 

\begin{Definition}\label{Def:KhovTangInv}
The Khovanov invariant of the diskular tangle $R\in\mathbb{T}^{(m_1,\cdots,m_k;n)}$ is the chain complex
\begin{equation}
	C_{\text{Kh}}(R)=C(R)[N_-],
\end{equation}
where $(C(R),d)$ is defined above.
\end{Definition}

\noindent\textbf{Gluing process:} Fix an $(m_1,...,m_k,n_i)$-diskular tangle $T$ and an $(n_1,...,n_r,p)$-diskular tangle $S$, for $1\leq i\leq r$. The composition (gluing) $T\circ_i S$ corresponds to the following theorem.
\begin{Theorem}\cite[Lemma 4.16]{LLS21}\label{Thm:Glue}
For diskular tangles $S$ and $T$ as above, there is an isomorphism  
	\begin{equation}
		C_{\text{Kh}}(T)\otimes_{\mathcal{H}^{n_i}} C_{\text{Kh}}(S)\rightarrow C_{\text{Kh}}(T\circ_iS)
	\end{equation}
as $(\mathcal{H}^{n_1}\otimes_{\mathbb{F}}\cdots\otimes_{\mathbb{F}}\mathcal{H}^{n_{i-1}}\otimes_{\mathbb{F}}\big(\mathcal{H}^{m_1}\otimes_{\mathbb{F}}\cdots \otimes_{\mathbb{F}}\mathcal{H}^{m_k}\big)\otimes_{\mathbb{F}}\mathcal{H}^{n_{i+1}}\otimes_{\mathbb{F}}\cdots\otimes_{\mathbb{F}}\mathcal{H}^{n_r},\mathcal{H}^{p})$-bimodules.
\end{Theorem}

Also, if we compose $(r_{i,1},\cdots, r_{i,l_i};m_i)$-diskular tangles $S_i$ for $i=1,...,k$ and $(m_1,...,m_k;n)$-diskular tangle $T$ to build $(r_{1,1},\cdots, r_{k,l_k};n)$-diskular tangle $T\circ(S_1,...,S_k)$, then we have isomorphism of chain complexes:
	\begin{equation}\label{Equ:GenComposeChain}
		C_{\text{Kh}}(T)\underset{{\mathcal{H}^{m_1}\otimes\cdots\otimes\mathcal{H}^{m_k}}}{\otimes} (C_{\text{Kh}}(S_1),...,C_{\text{Kh}}(S_k))\cong C_{\text{Kh}}(T\circ(S_1,...,S_k)),
	\end{equation}
as $(\mathcal{H}^{r_{1,1}}\otimes_{\mathbb{F}}\cdots\otimes_{\mathbb{F}}\mathcal{H}^{r_{k,l_k}},\mathcal{H}^{n})$-bimodules. Hence, we can write
\begin{equation}\label{Equ:BoxTenChain}
		C_{\text{Kh}}(T\circ(S_1,...,S_k))\cong C_{\text{Kh}}(T)\underset{{\mathcal{H}^{m_1}\otimes\cdots\otimes\mathcal{H}^{m_k}}}{\otimes} (C_{\text{Kh}}(S_1)\otimes\cdots\otimes C_{\text{Kh}}(S_k))
	\end{equation}

Given an elementary cobordism $\mathfrak{s}$ between diskular $(m_1,\cdots,m_k;n)$-tangles $R$ and $S$, it induces a $(\mathcal{H}^{m_1}\otimes_{\mathbb{F}}\cdots \otimes_{\mathbb{F}}\mathcal{H}^{m_k},\mathcal{H}^{n})$-bimodules morphism $C_{\text{Kh}}(\mathfrak{s}): C_{\text{Kh}}(R)\rightarrow C_{\text{Kh}}(S)$ with grading shift $(0,\chi(\mathfrak{s}))$.

\subsubsection{Functoriality of Khovanov}\label{SSec:KhovFunct}
Khovanov homology is \emph{functorial} up to a sign. That is to say, for an ambient isotopy relative to the boundary between cobordism $\Sigma$ and $\Sigma'$ in $\mathbb{R}^3\times[0,1]$, the maps induced on Khovanov homology are equal up to a sign.

\begin{Theorem}\cite{Jac04,Bar05,Kho06}\label{Thm:KhovFunct}
	Let $\Sigma,\Sigma'\subset\mathbb{R}^3\times[0,1]$ be two smooth cobordism from  $L_0$ to $L_1$. If they are isotopic relative to the boundary, then up to a sign they induce chain homotopic maps on the Khovanov chain complex. Hence, the induced map on homology satisfies
    \begin{equation}
        \text{Kh}(\Sigma)=\pm \text{Kh}(\Sigma'):\text{Kh}(L_0)\rightarrow\text{Kh}(L_1).
    \end{equation}
\end{Theorem}
We will not prove theorem \ref{Thm:KhovFunct} here and refer the reader to \cite{Jac04,Bar05,Kho06}, but we discuss the central ideas involved in the proof. The movies of isotopic cobordisms are related by finite sequence of movie moves. Hence, to prove the theorem, we should compute the chain homotopy between maps induced by movie moves. The result follows by showing the Khovanov chain complex of equivalent movies have only $\pm\text{Id}$ as chain maps.

\subsubsection{Obstruction to ribbon concordance}\label{SSec:NecCutRibbon}

This section is devoted to re-stating the main result form \cite{LZ19}. Their result provides an obstruction to ribbon concordance, from Khovanov homology. First, we recite a few notions related to the study of smooth surfaces using Khovanov homology.

Let $\mathcal{C}ob^2$ denote the cobordism category of links, with objects the links in $S^3$ and morphism consisting of smooth surfaces embedded in $S^3\times [0,1]$. Let $\mathcal{K}ob^2$ denote the  pre-additive $\mathbb{Z}$-linear category freely generated by $\mathcal{C}ob^2$. The category $\mathcal{K}ob^2$ has the same objects as $\mathcal{C}ob^2$, and morphisms from $L$ to $L'$ in $\mathcal{K}ob^2$ are finite formal linear combinations of cobordisms from $L$ to $L'$ in $\mathcal{C}ob^2$ and the composition is induced from the composition in $\mathcal{C}ob^2$. Also, let $\mathcal{K}ob^2_{\bullet}$ denotes the pre-additive category of \textit{dotted} cobordisms \cite[Section 11.2]{Bar05}. In this category objects are links in $S^3$. A morphism is a finite formal sum of dotted cobordisms $\mathcal{S}$, where a dotted cobordism means a properly embedded smooth surface possibly with boundary in $S^3\times[0,1]$ and finitely many dots (marked points) on its interior. 

\begin{Definition}\label{Def:Neck}
    Let $\Sigma$ be a cobordism and $h$ be a smoothly embedded 3-dimensional 1-handle $[-1,1]\times \mathbb{D}^2$ in $\mathbb{R}^3\times [0,1]$ such that $\mathfrak{n}=h\cap \Sigma$ is an embedded annulus $[-1,1]\times S^1\hookrightarrow \Sigma$ with the movie of $\mathfrak{n}$ is given by figure \ref{Fig:NeckCut} (a). We call $\mathfrak{n}$ a \emph{standard neck} on the cobordism $\Sigma$.
\end{Definition}
Figure \ref{Fig:NeckCut} shows the local movie of a standard neck on a cobordism. It is consists of two band sums such that the number of connected components on the left frame of the movie is equal to the number of the connected components on the right frame. Up to isotopy relative to the boundary of a cobordism $\Sigma$, we can transform any neck to a standard neck.

\begin{figure}[h]
\centering
    \begin{subfigure}[c]{0.5\textwidth}
	\def\svgwidth{0.75\textwidth}
	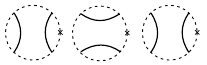

        \caption{}
    \end{subfigure}
    
    \begin{subfigure}[c]{0.45\textwidth}
	\def\svgwidth{0.8\textwidth}
	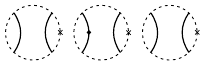

        \caption{}
    \end{subfigure}
    \begin{subfigure}[c]{0.45\textwidth}
	\def\svgwidth{0.8\textwidth}
	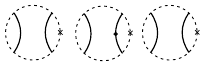

        \caption{}
    \end{subfigure}
	\caption[Neck cutting relation]{(a) Shows the movie of a standard neck. (b) and (c) show the movie of the cobordisms $\Sigma_+$ and $\Sigma_-$ respectively after the standard neck cut.\label{Fig:NeckCut}}
\end{figure}

Let $\Sigma_{+}$ (respectively $\Sigma_{-}$) be the result of deleting the neck $\mathfrak{n}\subset \Sigma$ from $\Sigma$ and smoothly gluing $\{-1,1\}\times \mathbb{D}^2\times \{0\} $ on the new boundary, and putting a new dot on $\{1\}\times \mathbb{D}^2 \times \{0\}$ (respectively $\{-1\} \times\mathbb{D}^2\times \{0\}$).

At the Khovanov homology package, a dot on a connected component of a cobordism induces the multiplication by $\boldsymbol{X}\in\mathcal{A}$ on the component with the dot and the identity on the rest of the components. Moreover, the placement of the dots on $\Sigma_+$ (respectively $\Sigma_-)$ does not affect the map $\text{Kh}(\Sigma_+)$ (respectively $\text{Kh}(\Sigma_-))$ induced on Khovanov homology. Hence analogous to the theorem \ref{Thm:KhovFunct}, we have,

\begin{Theorem}\cite{BLS17,Sar20}
	If $\Sigma$ and $\Sigma'$ are isotopic dotted cobordisms relative to the boundary in $S^3\times[0,1]$, then they are related by a sequence of movie moves and dot isotopies. Moreover, the maps induced by $\Sigma,\Sigma'$ are chain homotopic up to a sign.
\end{Theorem}
We will not prove this theorem here, but we will give some insight to how the proof works. A movie for a dotted cobordism $\Sigma$ consists of sequence of elementary cobordisms (a sequence of planar isotopies, Reidemeister moves, births, saddles, and deaths), and \emph{dot additions} that are frames in the movie with a dot appearing on an arc. Then, we can show that moving a dot under or over a crossing will induce chain homotopic maps on Khovano chain complex. Hence, the maps induced on Khovanov homology by isotopic dotted cobordisms are equal up to a sign. 

The Khovanov homology satisfies the following neck cutting relation.
\begin{Proposition}[Neck Cutting Relation]\cite[section 11.2]{Bar05}\label{Pro:KhoNecCut}
    Assume $\mathfrak{n}$ is a standard neck on the cobordism $\Sigma$, and $\Sigma_+$ and $\Sigma_-$ are the result of cutting the neck $\mathfrak{n}$. Then, 
    \begin{equation}\label{Equ:NCR}
    	\text{Kh}(\Sigma)= \pm\text{Kh}(\Sigma_+)\pm\text{Kh}(\Sigma_-)
    \end{equation}
\end{Proposition}
At the level of the Khovanov chain complex, the induced chain map $C_{\text{Kh}}(\Sigma):C_{\text{Kh}}(L_0)\rightarrow C_{\text{Kh}}(L_1)$ is up to a sign chain homotopic (more strongly, equal) to the chain map  $C_{\text{Kh}}(\Sigma_+)+C_{\text{Kh}}(\Sigma_-):C_{\text{Kh}}(L_0)\rightarrow C_{\text{Kh}}(L_1)$.

Lastly we have,
\begin{Theorem}\cite{LZ19}\label{Thm:Concordance}
    If $F$ is a smooth ribbon concordance from $L_0$ to $L_1$ , then the map induced on Khovanov homology is a bi-graded split injection.
\end{Theorem}

\subsection{Equivariant Khovanov homology}\label{Chp:EquKhovInv}
Section \ref{SSec:AlgBack} is devoted to studying the equivariant cohomology of chain complexes with a group action. We will direct our 
focus on Khovanov chain complex of periodic links in section \ref{Sec:EquiKhov}. Also, we fix the choice of our commutative ring 
$\mathbb{F}$ to be the integers $\mathbb{Z}$.

\subsubsection{Equivariant cohomology}\label{SSec:AlgBack}
Let $G$ be a finite cyclic group. We will denote by $EG$ the universal $G$-bundle. More precisely,  $EG$ is a contractible $G$-CW 
complex such that the action of $G$ is free. The orbit space $EG/G$ is the classifying space $BG$.Given $X$ a topological $G$-space, 
the $G$-\textit{equivariant cohomology} (also referred to as \textit{Borel equivariant cohomology}) 
of $X$ is defined as 
\begin{equation}
    H^{*}_{G}(X;\mathbb{Z}):=H^{*}(EG\times X/G;\mathbb{Z}), 
\end{equation}
where $(EG\times X)/G$ is the quotient of $EG\times X$ by the diagonal action of $G$. For more treatment of this topic, see \cite{Wei94,KSBrown82,BP17}. 

For computations, the cellular cochain complex $C^{*}_{\text{cell}}(EG;\mathbb{Z})$ can be considered as a free resolution of 
$\mathbb{Z}$ as a $\mathbb{Z}[G]$-module where $G$ acts on $\mathbb{Z}$ trivially. Additionally, the action of $G$ on $X$ makes 
the cochain complex $C^*_{\text{cell}}(X;\mathbb{Z})$ a $\mathbb{Z}[G]$-module. Hence, 
\begin{equation}
	H_{G}^*(X;\mathbb{Z}) = \textrm{Ext}^*_{\mathbb{Z}[G]}(\mathbb{Z}, C^{*}_{\text{cell}}(X;\mathbb{Z})).
\end{equation}

The $G$-equivariant cohomology of $X$ is equivalent to computing the group homology of $\mathbb{Z}$ with coefficients in 
$C^*_{\text{cell}}(X;\mathbb{Z})$ \cite[Chapter VII]{KSBrown82}. Instead of $\mathbb{Z}$, we could use any $\mathbb{Z}[G]$-module $M$.

With abuse of notation we denote the cellular chain complex $C^{\text{cell}}_{\bullet}(EG,\mathbb{Z})$ by $EG$. That is, we have a chain complex
\begin{equation}
    \begin{tikzcd}    
        EG:\: \cdots \arrow[r,"1-\theta"] & EG^{2} =\mathbb{Z}[G] \arrow[r,"1+\cdots+\theta^{p-1}"] &EG^{1}=\mathbb{Z}[G] \arrow[r,"1-\theta"] 
        & EG^{0}=\mathbb{Z}[G] \arrow[r] &0
    \end{tikzcd}
\end{equation}
where $\mathbb{Z}[G]=\mathbb{Z}[\theta]/(1-\theta^p)$, and $p$ is the order of the group $G$.

\begin{Proposition}
    Let $E{G}^*=\mathrm{Hom}_{\mathbb{Z}[G]}(EG,\mathbb{Z}[G])$, for a bounded below cochain complex $C$ over $\mathbb{Z}[G]$, 
    there is a natural isomorphism 
    \begin{equation}\label{Equ:EquiNatuIso}
        C\otimes_{\mathbb{Z}[G]}E{G}^*\rightarrow \mathrm{Hom}_{\mathbb{Z}[G]}(EG,C)
    \end{equation}
    of $\mathbb{Z}$-modules give by $c\otimes x\mapsto \phi_{c\otimes x}:EG\rightarrow C$ that is defined by $\phi_{c\otimes x}(a)=cx(a)$ 
    for $c\in C$, $ x\in EG^*$,  and $ a\in EG$.
\end{Proposition}
\begin{proof}
    It is immediate from the definition of equivariant cohomology.
\end{proof}

\subsubsection{Equivariant Khovanov homology}\label{Sec:EquiKhov}
Let $\text{Kh}(L)$ and $C_{\text{Kh}}(L)$ respectively denote the Khovanov homology and the Khovanov chain complex of a link $L$, 
with coefficients in $\mathbb{Z}$ (section \ref{SSec:KhovanovInv}). Also, to avoid introducing new notation, we denote by 
$\theta$ the generator of the group ring $\mathbb{Z}[\mathbb{Z}_p]=\mathbb{Z}[\theta]/(\theta^p-1)$. 

\begin{Definition}
The \emph{equivariant Khovanov homology} for a $p$-periodic link $L$ with coefficients in $\mathbb{Z}[\mathbb{Z}_p]$-module $M$ is defined as
\begin{equation}
		\mathrm{EKh}(L,M)=\mathrm{Ext}_{\mathbb{Z}[\mathbb{Z}_p]}(M,C_{\mathrm{Kh}}(L;\mathbb{Z})).
\end{equation}
\end{Definition}

Note that in literature despite the fact that Khovanov invariant is defined as a cochain complex, it has referred to as Khovanov homology. 
We also adhere to this inconsistency here, and refer to the equivariant Khovanov invariant defined above by equivariant Khovanov homology 
instead of equivariant Khovanov cohomology. Hence, for computations using the Khovanov chain complex,
\begin{equation}
		C_{\text{EKh}}^k(L,M)=\displaystyle{\bigoplus_{i+j=k}} \text{Hom}_{\mathbb{Z}[\mathbb{Z}_p]}(P_i,C_{\text{Kh}}^{j}(L))
\end{equation}
where $(P_{\bullet},\delta_{\bullet})$ is a projective resolution of $M$ as a $\mathbb{Z}[\mathbb{Z}_p]$-modules. Differentials for 
this complex are given by $d=\delta+d_{\text{Kh}}$. 

Since $C_{\text{Kh}}$ is a finitely generated free abelian group, if we fix $M=\mathbb{Z}$, we have,
\begin{equation}
	C_{\text{EKh}}^k(L;\mathbb{Z})=\bigoplus_{i+j=k}\mathbb{Z}[\mathbb{Z}_p]^*\underset{\mathbb{Z}[\mathbb{Z}_p]}{\otimes}C_{\text{Kh}}^i(L)
    \cong \bigoplus_{j\in \mathbb{Z}} C_{\text{Kh}}^{k-j}(L).
\end{equation}
 
Computing the equivariant Khovanov homology using periodic diskular tangles is done by decomposing the periodic diagram $D$ of the periodic link $L$, as 
\begin{equation}
    D = T\circ(S_1,...,S_p)
\end{equation}
where $T=T^{(n,\cdots,n;0)}$, $S_1=\cdots=S_p=T^{(;n)}$, and $\mathbb{Z}_p$ maps $T$ to itself, and acts by permutation on $(S_1,\cdots, S_p)$. Hence,
\begin{align}
    C_{\text{EKh}}(D) &= E\mathbb{Z}_p\underset{\mathbb{Z}[\mathbb{Z}_p]}{\bigotimes}\bigg(C_{\text{Kh}}(T)\underset{\mathcal{H}^n
    \otimes\cdots\otimes \mathcal{H}^n}{\otimes} C_{\text{Kh}}(S_1,...,S_p)\bigg)\\
\end{align}
where the tensor product on the right hand side is by the gluing lemma \ref{Thm:Glue}.

\subsubsection{Skew group ring}\label{SSec:TwiGroRin}
In this section we are going to study the algebraic structure of the Khovanov invariant of an admissible $p$-periodic diskular tangle. 
In summary, we show the Khovanov chain complex associated to a $p$-equivariant diskular tangle is projective over a Khovanov arc algebra 
(definition \ref{Def:ArcAlg}) equipped with the $\mathbb{Z}_p$ action. Here unless otherwise stated, tensor product $\otimes$ denotes 
the tensor product over $\mathbb{Z}$.

\begin{Definition}\label{Def:SkeGroRing}
Let $R$ be a unital ring, $G$ a finite group and $\theta:G\rightarrow\mathrm{Aut}_{\text{Ring}}(R)$ a group homomorphism. The 
\emph{skew group ring} of $G$ over $R$ induced by $\theta$ as a left $R$-module is given by
\begin{equation}\label{Equ:SkeGroRing}
    R_{\theta}[G]=\bigoplus_{g\in G}R\{g\},
\end{equation}
where $g\in G$ is considered as a formal variable and $\mathrm{Aut}_{\text{Ring}}(R)$ denotes the group of ring automorphisms of $R$. 
The addition in $R_{\theta}[G]$ is component-wise and multiplication is given by $rg\cdot sh= r\theta(g)(s)gh\in R\{gh\}$ for 
$rg\in R\{g\}$ and $sh\in R\{h\}$.
\end{Definition}

To simplify the notation, we denote $R_g=R\{g\}$, and we will write $\theta_g(b)$ to denote the image of $b\in R$ by the ring 
isomorphism $\theta(g):R\rightarrow R$. Also, we can see elements of $R_{\theta}[G]$ as formal sums
\begin{equation}
    R_{\theta}[G]=\{\sum_{g\in G} r_gg\vert r_g\in R_g\}.
\end{equation}

\begin{Proposition}\label{Prop:SkeGruRing}
    The skew group ring $R_{\theta}[G]$ defined above is a unital associative ring with multiplicative identity given by $1=1_R1_G\in R_{1_G}$.
\end{Proposition}
The proof is straightforward using the associativity of $R$ and the fact that $\theta$ is a group homomorphism. The skew group ring belongs 
to a more general class of rings known as \emph{strongly group graded rings}. We encourage the reader to look at \cite{Dad80,CK96,BG00} 
for more details on strongly group graded ring.

\begin{Example}
Assume $R=\mathbb{Z}$ and $G=\mathbb{Z}_p$ for $p$ a positive integer. Let $\theta:\mathbb{Z}_p\rightarrow \mathrm{Aut}_{\text{Ring}}(\mathbb{Z})$ 
be defined by 
\begin{equation}
    \theta(g)=id_{\mathbb{Z}} \quad \text{for all }g\in \mathbb{Z}_p.
\end{equation}
Then $\mathbb{Z}_{\theta}[\mathbb{Z}_p]$ is the usual group ring $\mathbb{Z}[\mathbb{Z}_p]$.
\end{Example}

Another example of a skew group ring is the following.
\begin{Example}
Let $A$ be a unital ring and assume $G\leqslant S_p$ is a subgroup of the $p$th symmetric group for a positive integer $p$. We have an action of 
$S_p$ on $A^{\otimes p}=A{\otimes}\cdots {\otimes}A$ by permutation of factors. Therefore, $G$ also acts on $A^{\otimes p}$. We can define a 
multiplication for the $A$-module $A^{\otimes p}{\otimes} \mathbb{Z}[G]$ by 
\begin{equation}\label{Equ:wreathprod}
    (a \otimes g)(b \otimes h) = a\cdot g(b)\otimes gh
\end{equation}
for $a, b \in A^{\otimes p}$ and $g, h \in G$. The resulting $\mathbb{Z}$-algebra $A^{\otimes p}{\otimes} \mathbb{Z}[G]$ is call the 
\emph{wreath product} of $A$ with $G$, and we denote it by $A\wr G$. 
\end{Example}

Let us fix an $R_{\theta}[G]$-module $M$ and $\mathbb{Z}[G]$-module $V$. Then $M\otimes V$ is naturally a left $R_{\theta}[G]\otimes\mathbb{Z}[G]$-module. 
Additionally, one can make $V\otimes M$ a left $R_{\theta}[G]$-module as follows. The \emph{semi-diagonal action} of $R_{\theta}[G]$ on $M\otimes V$ is defined by 
\begin{equation}\label{Equ:SemiDiagAct}
    rg\cdot(v\otimes m) = gv\otimes r\cdot g(m)
\end{equation}
This $R_{\theta}[G]$ action makes $M\otimes V$ a left $R_{\theta}[G]$-module. Again the proof of this statement is left as an exercise. 

Now we can compare the projective modules over $R_{\theta}[G]$ with $R$-modules.
\begin{Lemma}\cite[lemma 4.1]{BG00}\label{Lem:SemDiaTenIsProj}
Let $V$ be a left $\mathbb{Z}[G]$-module, and $M$ a left
$R_{\theta}[G]$-module. Assume that $V$ is free as a $\mathbb{Z}$-module. If $M$ is projective as an $R_{\theta}[G]$-module, then so is 
$V\otimes M$ as a left $R_{\theta}[G]$-module with the semi-diagonal action.
\end{Lemma}

\begin{Lemma}\cite[lemma 4.3]{BG00}\label{Lem:RProjIsRGProj}
Let $M$ be an $R_{\theta}[G]$-module. If $M$ is projective as a left $R$-module, then $\mathbb{Z}[G]\otimes M$ with semi-diagonal action is 
projective as a left $R_{\theta}[G]$-module.
\end{Lemma}

Both lemma \ref{Lem:SemDiaTenIsProj} and \ref{Lem:RProjIsRGProj} are proven for strongly group graded rings in \cite{BG00} of which skew group rings are a special case.

Now we will focus our attention to the special cases $R=\mathcal{H}^n$ (see definition \ref{Def:ArcAlg}), and $G=\mathbb{Z}_p$. 

\begin{Lemma}\cite[Proposition 3]{Kho02}\label{Lem:KhoTanProj}
Let $S \in \mathbb{T}^{(;n)}$ be an $(;n)$-diskular tangle. The $\mathcal{H}^n$-module $C_{\text{Kh}}(S)$ is a finitely generated and projective 
$\mathcal{H}^n$-module.
\end{Lemma} 
\begin{proof}
By the correspondence between the tangles diagrams in $[0,1]^2$ and diskular tangles (section \ref{SSec:DiskTang}), we can consider the $(;n)$-diskular 
tangle $S$ as a tangle diagram in $[0,1]^2$. Hence, the lemma follows from \cite[Proposition 3]{Kho02}.
\end{proof}

For $p$-equivariant diskular tangle $(S_1,...,S_p)$ with $S_1=\cdots=S_p\in T^{(;n)}$ (see definition \ref{Def:EquiDiskTang}) we have,

\begin{Corollary}\label{Prop:EquiKhovTanProj}
    The Khovanov chain complex of $C_{\text{Kh}}(S_1,...,S_p)=C_{\text{Kh}}(S_1)\otimes\cdots \otimes C_{\text{Kh}}(S_p)$ is projective as an 
    $\mathcal{H}^n\otimes \cdots\otimes \mathcal{H}^n$-module.
\end{Corollary}
\begin{proof}
	This is immediate from lemma \ref{Lem:KhoTanProj}. 
\end{proof}

The cyclic group $\mathbb{Z}_p$ also acts on both $\mathcal{H}^n\otimes\cdots\otimes \mathcal{H}^n$ and 
$C_{\text{Kh}}(S_1,...,S_p)= C_{\text{Kh}}(S_1)\otimes\cdots\otimes C_{\text{Kh}}(S_p)$ by cyclic permutation, denoted by $\theta$. 
Hence, we can consider $C_{\text{Kh}}(S_1)\otimes\cdots\otimes C_{\text{Kh}}(S_p)$ as a module over 
$\mathcal{R}^n_{\theta}=(\mathcal{H}^n\otimes\cdots\otimes \mathcal{H}^n)_{\theta}[\mathbb{Z}_p]$. 

\begin{Theorem}\label{Thm:KhoTwiProj}
    Given a $p$-equivariant $(;n)$-diskular tangle $(S_1,...,S_p)$, the $\mathcal{R}^n_{\theta}$-module 
    $\mathbb{Z}[\mathbb{Z}_p]\otimes C_{\text{Kh}}(S,...,S)$ with the semi-diagonal action is a projective $\mathcal{R}^n_{\theta}$-module.
\end{Theorem}
\begin{proof}
    Proof follows immediately from corollary \ref{Prop:EquiKhovTanProj} and lemma \ref{Lem:RProjIsRGProj}.
\end{proof} 

\section{Functoriality of Equivariant Khovanov Homology}\label{Chp:Result}
In this section, we study the maps induced on equivariant Khovanov homology by equivariant cobordisms. 
In section \ref{SSec:EquiFunct}, we prove equivariant Khovanov homology is functorial up to a factor of 
$(\pm 1)^p$ where $p$ is the order of the group. Using that result, we study the map induced by a ribbon 
concordance in section \ref{Sec:EquConcObs}.

\subsection{Equivariant movie moves}
Here, we introduce movie presentations for equivariant cobordisms between periodic links. Also, 
we prove that for equivariantly isotopic cobordisms, their movies are related by equivariant movie moves. 
In what follows, $p$ is a fixed prime integer, and is the order of the group. In an overview, we use 
techniques similar to section \ref{SSec:KhovFunct}.  

\begin{Definition}\label{Def:EquMovie}
    A $p$-\emph{equivariant movie} of an equivariant link cobordism $\Sigma: L_0\rightarrow L_1$, is a 
    finite sequence of periodic link diagrams $\boldsymbol{EM}_{\Sigma}=\{E_i\}_{i=0}^k$, with successive 
    pairs of diagrams related by a $p$-equivariant elementary cobordism (definition \ref{Def:EquEleCob}) 
    localized to $p$ disjoint closed disks in $\mathbb{R}^2\backslash\boldsymbol{0}$ such that the elementary 
    cobordisms correspond under the action of $\mathbb{Z}$. Individual diagrams $E_i$ in the sequence are 
    called \emph{equivariant frames}.
\end{Definition}

Let $\pi:(\mathbb{R}^3\backslash\boldsymbol{z})\times[0,1]\rightarrow \mathbb{R}^2\backslash\boldsymbol{0}$ 
denotes the projection to the $xy$-plane. Fix a smooth equivariant cobordism $\Sigma: L_0\rightarrow L_1$, between periodic links 
$L_0$ and $L_1$. We call $\Sigma$ \emph{generic}, if the projection 
$\rho:(\mathbb{R}^3\backslash\boldsymbol{z})\times [0,1]\rightarrow[0,1]$ restricted to $\Sigma$ is a Morse function.

\begin{Lemma}\label{Lem:EquMovie}
    Any equivariant link cobordism $\Sigma$ can be represented by an equivariant movie.
\end{Lemma}
\begin{proof}
Let $\overline{\Sigma}$ be the quotient cobordism in $(\mathbb{R}^3\backslash\boldsymbol{z})\times [0,1]$. After small perturbations of 
$\overline{\Sigma}$ away from the $\boldsymbol{z}\times [0,1]$, $\overline{\Sigma}$ is a generic cobordism. By \cite[Lemma 15]{GLW18}, 
the link cobordism $\overline{\Sigma}\subset \mathbb{R}^3\times [0,1]$ can be presented by an annular movie 
$\boldsymbol{M}_{\overline{\Sigma}}=\{\overline{E}_{t_i}\}$ (A movie in which every frame avoids the origin). Pulling back the movie 
$\boldsymbol{M}_{\overline{\Sigma}}$ with the quotient map induces an equivariant movie $\boldsymbol{M}_{\Sigma}$ as desired.
\end{proof}

An \emph{equivariant isotopy} refers to a 1-parameter family of diffeomorphisms $f_s$ of 
$(\mathbb{R}^3\backslash\boldsymbol{z})\times[0,1]$ for $s\in [0,1]$, such that $f_0=id$ and $f_s$ is the identity on the boundary 
for all $s$, and equivariant under the action $\tilde{\theta}$. Hence, for all $s$, the image of $f_s$ restricted to $\Sigma$ is 
a $p$-equivariant link cobordism.

Fix periodic links $L_0$ and $L_1$ and equivariant cobordisms $\Sigma, \Sigma':L_0\rightarrow L_1$. We call $\Sigma$ and 
$\Sigma'$ \emph{equivariantly isotopic}, if there is an equivariant isotopy $f_s$ such that $f_1(\Sigma)=\Sigma'$.

 \begin{figure}[ht]
	\centering
	\def\svgwidth{1\textwidth}
	%% Creator: Inkscape 1.2.2 (b0a8486, 2022-12-01), www.inkscape.org
%% Accompanies image file 'EquiMM.pdf' (pdf, eps, ps)
\begingroup%
  \makeatletter%
  \providecommand\color[2][]{%
    \errmessage{(Inkscape) Color is used for the text in Inkscape, but the package 'color.sty' is not loaded}%
    \renewcommand\color[2][]{}%
  }%
  \providecommand\transparent[1]{%
    \errmessage{(Inkscape) Transparency is used (non-zero) for the text in Inkscape, but the package 'transparent.sty' is not loaded}%
    \renewcommand\transparent[1]{}%
  }%
  \providecommand\rotatebox[2]{#2}%
  \newcommand*\fsize{\dimexpr\f@size pt\relax}%
  \newcommand*\lineheight[1]{\fontsize{\fsize}{#1\fsize}\selectfont}%
  \ifx\svgwidth\undefined%
    \setlength{\unitlength}{802.20472441bp}%
    \ifx\svgscale\undefined%
      \relax%
    \else%
      \setlength{\unitlength}{\unitlength * \real{\svgscale}}%
    \fi%
  \else%
    \setlength{\unitlength}{\svgwidth}%
  \fi%
  \global\let\svgwidth\undefined%
  \global\let\svgscale\undefined%
  \makeatother%
  \begin{picture}(1,0.27208481)%
    \lineheight{1}%
    \setlength\tabcolsep{0pt}%
    \put(0,0){\includegraphics[width=\unitlength,page=1]{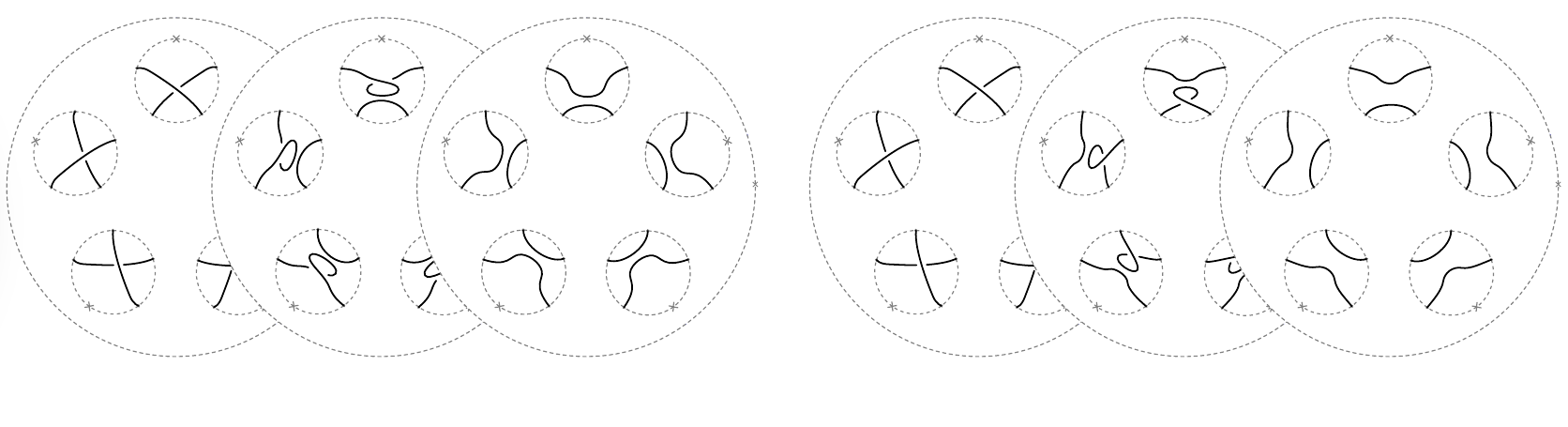}}%
    \put(0.49608028,0.14777389){\color[rgb]{0,0,0}\makebox(0,0)[t]{\lineheight{1.25}\smash{\begin{tabular}[t]{c}$\cong$\end{tabular}}}}%
  \end{picture}%
\endgroup%

	\vspace{-0.7cm}\caption[Example of an equivariant movie move]{This figure shows an equivariant movie move. The equivariant movie on the 
    left is equivalent to the equivariant movie on the right. These two movies differ by commutativity of a saddle move and 
    Reidemeister move of type I.}
    \label{Fig:EquDiskMM}
\end{figure}

Now we can prove the theorem \ref{Thm:EquiMM}.

\begin{proof}[Proof of theorem \ref{Thm:EquiMM}] By definition \ref{Def:EquiCob}, if $\Sigma$ is an equivariant link cobordism 
    from $L_0$ to $L_1$, its quotient $\overline{\Sigma}$ by the extended action $\tilde{\theta}$ on $(\mathbb{R}^3\backslash\boldsymbol{z})\times[0,1]$ 
    is an annular link cobordism $\overline{\Sigma}\subset (\mathbb{R}^3\backslash \boldsymbol{z})\times [0,1]$. 
By assumption $\Sigma$ and $\Sigma'$ are two equivariant cobordisms. Their quotient cobordisms $\overline{\Sigma}$ and $\overline{\Sigma'}$ 
are isotopic annular link cobordisms in $(\mathbb{R}^3\backslash \boldsymbol{z})\times[0,1]$. By \cite[Lemma A.2]{GLW18}, $\overline{\Sigma}$ 
and $\overline{\Sigma'}$ are related by a finite sequence of movie moves localized to a disks away from the origin in $\mathbb{R}^2$. 
The pull back of those movie moves are a sequence of equivariant movie moves from $\boldsymbol{EM}_{\Sigma}$ to $\boldsymbol{EM}_{\Sigma'}$.
\end{proof}

\subsection{Functoriality of Equivariant Khovanov Homology}\label{SSec:EquiFunct}
In this section, we prove theorem \ref{Thm:EquiFunct}. First we mention a few abstract homological algebra facts. In what follows, unless 
otherwise stated, we use $\mathbb{Z}$ as our coefficient ring.

\begin{Lemma}\label{Lem:ChaMapH1}
    Suppose $C$, $D$ are chain complexes, $f_0, f_1:C\rightarrow D$ are chain homotopic maps by a chain homotopy $h$. This information is 
    equivalent to a chain map $H:I\otimes C\rightarrow D$, where $I$ is the chain complex
    \begin{equation}
        \begin{tikzcd}
            I^{1}=\mathbb{Z}\{ s\} \arrow[r] & I^{0}=\mathbb{Z}\{s_0,s_1\}\\
            s \arrow[r,mapsto] &s_1-s_0,
        \end{tikzcd}
    \end{equation}
    such that $H(s_i\otimes x)=f_i(x)$, for $i=0,1$, moreover, $H(s\otimes x)=h(x)$.
\end{Lemma}
The proof of the lemma above is trivial and left as an exercise for the reader. Given $f_0, f_1$ as above, then 
$f_0^{\otimes p}, f_1^{\otimes p}: C^{\otimes p}\rightarrow D^{\otimes p}$ are also chain homotopic with the chain homotopy given by
\begin{equation}\label{Equ:TenChaHomtp}
    k=\sum_{i=0}^{p-1}\pm f_1^{\otimes i}\otimes h\otimes f_0^{\otimes p-1-i}: C^{\otimes p}\rightarrow D^{\otimes p}[1]
\end{equation}
where the sign comes from the Koszul sign convention for tensor product of chain complexes.

Similarly, we introduce a chain map using the chain map $H$ (lemma \ref{Lem:ChaMapH}) as follows.
\begin{align}\label{Equ:ChaMapK}
    K&:(I\otimes C)^{\otimes p} \longrightarrow D^{\otimes p}\\\nonumber
    &K(t_1\otimes x_1\otimes \cdots \otimes t_p\otimes x_p) = H(t_1\otimes x_1)\otimes \cdots \otimes H(t_p\otimes x_p) \\\nonumber
\end{align}
Equivalently, we can rewrite the equation (\ref{Equ:ChaMapK}) as $K:I^{\otimes p}\bigotimes C^{\otimes p}\longrightarrow D^{\otimes p}$.

\begin{Lemma}\label{Lem:ChaMapExt}
    Assume $(C,d)$ and $(C',d')$ are chain complexes of $R$-modules and $n$ is an integer. Let $\{g_i:C_i\rightarrow C'_{i}\}$ for $i\leq n$ be a family of group homomorphisms such that $g_i\circ d=d'\circ g_i$. If for $i\geq n$, the $C_i$ are projective and $H_i(C')=0$, then we can extended the $g_i$ to a chain map $g:C\rightarrow C'$. Moreover, $g$ is unique up to chain homotopy.
\end{Lemma}
\begin{proof}
By induction, we can assume for $k\geq n$ we have extended the $g_i$ for $i\leq k$. For $k+1$ we have a commutative diagram 
\begin{equation}
    \begin{tikzcd}
    C_{k+1}\arrow[r,"d"]\arrow[d,dashed, "g_{k+1}"] & C_k \arrow[r,"d"]\arrow[d,"g_k"] & C_{k-1} \arrow[d,"g_{k-1}"]\\
    C'_{k+1}\arrow[r,"d'"] & C'_{k} \arrow[r,"d'"] & C'_{k-1}\\
\end{tikzcd}
\end{equation}
where $d'\circ g_k\circ d=g_{k-1}\circ d\circ d=0$. As $C_{k+1}$ is projective and the bottom row is exact, the map $g_{k+1}$ (dashed arrow) exists. 
Suppose now that $\mathfrak{g}:C\rightarrow C'$ is a second extension of $\{g_i\}_{i\leq n}$. Assume by induction the  $\phi_i:C_i\rightarrow C'_{i+1}$ have been defined for $i\leq k$, where $k\geq n$, and  $d'h_{i}+h_{i-1}d=g_i-\mathfrak{g}_i$.
\begin{equation}\label{Equ:ComDiaPar}
    \begin{tikzcd}[column sep=large]
    & C_{k+1} \arrow[r,"d"]\arrow[d,"g_{k}-\mathfrak{g}_{k}" near start]\arrow[ld,dashed,"\phi_{k+1}"] & C_{k} \arrow[d,"g_{k-1}-\mathfrak{g}_{k-1}" near start]\arrow[r,"d"]\arrow[ld,"\phi_{k}"]& C_{k-1}\arrow[ld,"\phi_{k}"]\\
    C'_{k+2}\arrow[r,"d'"'] & C'_{k+1} \arrow[r,"d'"'] & C'_{k}&\\
\end{tikzcd}
\end{equation}
By assumptions, we have
\begin{align*}
    d'\circ \phi_k\circ d &= (g_{k}-\mathfrak{g}_k-\phi_{k-1}\circ d)\circ d\\
    &= g_{k}\circ d-\mathfrak{g}_k\circ d-\phi_{k-1}\circ d\circ d\\
    &= d'\circ g_{k}-d'\circ \mathfrak{g}_k=d'\circ (g_{k}- \mathfrak{g}_k)\\
\end{align*}
As $H_i(C')=0$ for $i\geq n$, the bottom row in diagram \ref{Equ:ComDiaPar} is exact. Also, $C_k$ is projective, and this implies the desired map $\phi_{k+1}$ exists.
\end{proof}

We define an action of $\mathbb{Z}_p$ on $C^{\otimes p}$ by cyclic permutation of factors i.e., $\mathbb{Z}_p$ acts by 
\begin{equation}
    \theta\cdot( x_1\otimes \cdots \otimes x_p)=(-1)^{|x_p|(|x_1|+\cdots |x_{p-1}|)}x_p\otimes x_1\otimes \cdots \otimes x_{p-1},
\end{equation} 
where $|x_i|=\text{gr}_h(x_i)$ denotes the homological degree of the $x_i\in C$. Note that the homotopy in equation (\ref{Equ:TenChaHomtp}) is not equivariant.

For a ring $R$, consider the category $\textbf{Ch}_{R}$ of chain complexes of $R$-modules. The morphisms of $\textbf{Ch}_{R}$ are chain maps. The homotopy category $\mathcal{K}$ of $\textbf{Ch}_{R}$ is defined as follows: The objects of $\mathcal{K}$ are chain complexes (the same as the objects of $\textbf{Ch}_{R}$) and the morphisms of $\mathcal{K}$ are the equivalence classes of chain maps in $\textbf{Ch}_{R}$ up to chain homotopy. One can check that $\mathcal{K}$ is well defined as a category. By $\mathcal{D}$, we denote the \textit{derived category} of $\mathcal{K}$ defined to be the localization $Q^{-1}\mathcal{K}$ at the collection $Q$ of quasi-isomorphisms.

\begin{Lemma}\cite[Corollary 10.4.7]{Wei94}\label{Lem:HomAndDer}
For a bounded above chain complex $M$ of projective $R$-modules, there is an isomorphism
    \begin{equation}
        \text{Hom}_{\mathcal{K}}(M, N)\cong \text{Hom}_{\mathcal{D}}(M,N)
    \end{equation}
for every $R$-module $N$.
\end{Lemma}

%A chain complexes $(C,d)$ with $d_i:C_i\rightarrow C_{i-1}$ is \textit{bounded above} if there exists an $n\in \mathbb{Z}$ such that $C_i=0$ for all $i<n$.

To construct an equivariant chain homotopy between chain maps $f_0^{\otimes p}$ and $f_1^{\otimes p}$, we need to define a chain map
\begin{equation}\label{Equ:EquiChaHom}
    \widehat{H}: E\mathbb{Z}_p\otimes I \longrightarrow I^{\otimes p}
\end{equation}
satisfying
\begin{equation}
    \widehat{H}_k(\alpha\otimes w) = \alpha\cdot \widehat{H}_k(1\otimes w) \nonumber
\end{equation}
for $w\in(I^{\otimes p})_k$, $\alpha\in\mathbb{Z}[\mathbb{Z}_p]$, and $k\geq 0$. First, we define the following morphism at homological degree $k=0,1$. 
\begin{align}
    &\widehat{H}_0(1\otimes s_j) = s_j\otimes \cdots \otimes s_j \quad\quad (\text{for } j=0,1)\nonumber\\
    &\widehat{H}_1(1\otimes s_j) = 0 \quad\quad (\text{for } j=0,1)\nonumber\\
    &\widehat{H}_1(1\otimes s) =\sum_{i=1}^{p}s_1^{\otimes p-i-1}\otimes s \otimes s_0^{\otimes i}\nonumber\\
    &\widehat{H}_k(\theta^l\otimes w) = \theta^l\cdot \widehat{H}_k(1\otimes w) \quad (\forall w\in I,\; k=0,1, \text{ and } l=0,\ldots, p-1)\nonumber
\end{align}
where $\theta$, acts by cyclic permutation. We use the partial information at $k=0,1$ to extend $\widehat{H}$ to every homological degree by lemma \ref{Lem:ChaMapExt}.
\begin{Lemma}\label{Lem:ChaMapH}
The map $\widehat{H}$ defined above extends to a chain map.
\end{Lemma}
\begin{proof}
    At homological degree $k=0,1$, $\widehat{H}$ commutes with the differentials.
\begin{align}
        d\circ \widehat{H}_{1} (1\otimes s_j) &= 0 =  \widehat{H}_{0}\circ d(1\otimes s_j)\quad\quad (\text{for } j=0,1)\nonumber\\
        d\circ \widehat{H}_1(\theta^i\otimes s_j) &=0=\widehat{H}_{0}\circ d(\theta^i\otimes s_j)\quad\quad (\text{for } j=0,1)\nonumber\\
        d\circ \widehat{H}_{1} (1\otimes s) &= d(\sum_{i=0}^{p-1}s_1^{\otimes p-i-1}\otimes s \otimes s_0^{\otimes i}) 
        =\sum_{i=0}^{p-1} s_1^{\otimes p-i-1}\otimes (s_1-s_0) \otimes s_0^{\otimes i}\nonumber\\
        &=s_1^{\otimes p} -s_0^{\otimes p}\nonumber\\
        & =\widehat{H}_0(1\otimes s_1-1\otimes s_0)= \widehat{H}_0\circ d(1\otimes s)\nonumber
\end{align} 
for $i=1,\cdots p-1$. The same relation will be hold for $\widehat{H}_1(\theta^i\otimes s)$ for $i=1,\cdots,p-1$. By definition, $(E\mathbb{Z}_p\otimes I)_k$ 
is a projective $\mathbb{Z}[\mathbb{Z}_p]$-module for $k\geq 0$, and $H_k(I^{\otimes p})=0$ for $k\geq 1$. 
Hence by lemma \ref{Lem:ChaMapExt}, we can extend the maps $\widehat{H}_1$ and $\widehat{H}_0$ to a chain map $\widehat{H}=\{\widehat{H}_k\}$ inductively.
\end{proof}

\begin{Example}
    As an example, we have written down the chain map $\widehat{H}$ for $p=2,3$. 
    \begin{itemize}
        \item For $p=2$, we can define $\widehat{H}$ as follows.
        \begin{align*}
        \widehat{H}:&E\mathbb{Z}_2\otimes I\rightarrow I^{\otimes 2}\\
        &\widehat{H}_0(a\otimes s_j)=s_j\otimes s_j\\
        &\widehat{H}_{1}(a\otimes s_j) = 0 \\
        &\widehat{H}_{1}(1\otimes s) = s_1\otimes s + s\otimes s_0\\
        &\widehat{H}_{1}(\theta\otimes s) = \theta\cdot \widehat{H}_{1}(1\otimes s)=s\otimes s_1 + s_0\otimes s\\
        &\widehat{H}_{2}(a\otimes s_j) = 0 \\
        &\widehat{H}_{2}(1\otimes s) = s\otimes s\\
        &\widehat{H}_{2}(\theta\otimes s) = \theta\cdot \widehat{H}_2(1\otimes s)=-s\otimes s 
        \end{align*}
        for $j=0,1$ and $a=1,\theta$. Moreover, we define $\widehat{H}_k=0$ for $k>2$.
    \item For $p=3$, we can define $\widehat{H}$ as follows.
        \begin{align*}
        \widehat{H}:&E\mathbb{Z}_3\otimes I\rightarrow I^{\otimes 3}\\
            &\widehat{H}_0(a\otimes s_j) =s_j\otimes s_j \otimes s_j \\
            &\widehat{H}_{1}(a\otimes s_j) = 0 \\
            &\widehat{H}_{1}(1\otimes s) = s_1\otimes s_1\otimes s + s_1\otimes s\otimes s_0 + s\otimes s_0\otimes s_0\\
            &\widehat{H}_{1}(\theta\otimes s) = \theta\cdot H(1\otimes s) = s_0\otimes s_1\otimes s + s_0\otimes s\otimes s_0 + s\otimes s_1\otimes s_1\\
            &\widehat{H}_{1}(\theta^2\otimes s) = \theta^2\cdot H(1\otimes s) = s_0\otimes s_0\otimes s + s_1\otimes s\otimes s_1 + s\otimes s_0\otimes s_1\\
            &\widehat{H}_{2}(a\otimes s_j) = 0 \\
            &\widehat{H}_{2}(1\otimes s) =  s\otimes s_1 \otimes s + s\otimes s \otimes s_0\\
            &\widehat{H}_{2}(\theta\otimes s) = \theta\cdot \widehat{H}_{2}(1\otimes s) = s_0\otimes s \otimes s - s\otimes s \otimes s_1\\
            &\widehat{H}_{2}(\theta^2\otimes s) = \theta^2\cdot \widehat{H}_{2}(1\otimes s) = - s_1\otimes s \otimes s - s\otimes s_0 \otimes s\\
            &\widehat{H}_{3}(a\otimes s_j) = 0\\
            &\widehat{H}_{3}(1\otimes s) = \widehat{H}_{3}(\theta\otimes s)=  \widehat{H}_{3}(\theta^2\otimes s)=- s\otimes s \otimes s 
        \end{align*}
        for $j=0,1$ and $a=1,\theta,\theta^2$. Moreover, we define $\widehat{H}_k=0$ for $k>3$.
    \end{itemize}
\end{Example}

\begin{Proposition}\label{Pro:EquMapK}
    Assume $C,D, f_0$, and $f_1$ are as above. We can define a chain map as follows. 
\begin{equation}\label{Equ:EquChaMapK}
\begin{tikzcd}
     \widehat{K}:E\mathbb{Z}_p\otimes I\otimes C^{\otimes p}\arrow[r,"\widehat{H}\otimes id"] & (I\otimes C)^{\otimes p}\arrow[r,"K"]& D^{\otimes p}
\end{tikzcd}
\end{equation}
where $K$ is defined in \ref{Equ:ChaMapK}. Moreover, $K$ is an morphism of $\mathbb{Z}[\mathbb{Z}_p]$-modules.
\end{Proposition}
\begin{proof}
    This statement follows from the constructions above.
\end{proof}

One last algebraic fact we need is the following lemma.

\begin{Lemma}\label{Lem:SwitchTensor}
    As chain complex of Abelian groups, we have,
    \begin{equation}\label{Equ:SwitchTensor}
        C_{\text{Kh}}(R)\underset{\mathcal{R}^n_{\theta}}{\otimes} (E\mathbb{Z}_p\otimes C^{\otimes p}) \cong E\mathbb{Z}_p\underset{\mathbb{Z}[\mathbb{Z}_p]}{\otimes}\big(C_{\text{Kh}}(R)\underset{(\mathcal{H}^n)^{\otimes p}}{\otimes}C^{\otimes p}\big),
    \end{equation}
    where $C=C_{\text{Kh}}(T)$ for a $(;n)$-tangle $T$, and $R$ is a $(n\ldots,n;0)$-admissible tangle.
\end{Lemma}
\begin{proof}
    We can define a bijection by $n\otimes v\otimes m\mapsto v\otimes n\otimes m$. We have to check that this bijection respects the multi-linear relations on the left and right hand side of equation \ref{Equ:SwitchTensor}. Let us fix $r\in (\mathcal{H}^n)^{\otimes p}$, $m\in C^{\otimes p}$, $n\in C_{\text{Kh}}(R)$, and a generator $\theta^i\in\mathbb{Z}[\mathbb{Z}_p]$ for $i\in \{0,\ldots,p-1\}$. On the right hand side of (\ref{Equ:SwitchTensor}) we have relations,
    \begin{equation}\label{Equ:RightHandSide1}
    	v\otimes n\otimes r m = v\otimes n r\otimes m,
    \end{equation}
	and
    \begin{align}
    v\cdot\theta^i\otimes n\otimes m &= v\otimes \theta^i\cdot(n\otimes m) \label{Equ:RightHandSide}\\
    &= v\otimes \theta^i\cdot n \otimes \theta^i\cdot m.\nonumber
    \end{align}
	On the left hand side of (\ref{Equ:SwitchTensor}), we have,
    \begin{align}\label{Equ:LeftHandSide}
     n\cdot (r\theta^{i}) \otimes v\otimes m &= n\otimes (r\theta^i)\cdot  (v\otimes m) = n\otimes \theta^i\cdot  v\otimes r\theta^i\cdot m,
    \end{align}
    where the second equality in the above, is by the semi-diagonal $\mathcal{R}^n_{\theta}$-module structure on {$E\mathbb{Z}_p\otimes C^{\otimes p}$}. Moreover,  $C_{\text{Kh}}(R)$ is a right $\mathcal{R}^n_{\theta}$-module and $C^{\otimes p}$ is a left $\mathcal{R}^n_{\theta}$-module but a left $\mathbb{Z}[\mathbb{Z}_p]$-module and the $\mathcal{R}^n_{\theta}$ action is defined by
    \begin{equation}
    	n\cdot (r\theta^{i})=\theta^{-i}\cdot (n r)
    \end{equation}
    
     Hence, for the left term in the equation (\ref{Equ:LeftHandSide}), we have
    \begin{align*}
       \theta^{-i}\cdot (n r) \otimes v\otimes m &\mapsto  v\otimes \theta^{-i}\cdot (n r) \otimes m\\
       &\quad\: =  v\otimes \theta^{-i}\cdot (n r) \otimes \theta^{-i}\cdot(\theta^{i}\cdot m)\\
       &\quad\: =  v\otimes \theta^{-i}\cdot (n r \otimes \theta^{i}\cdot m)\\
       &\quad\: =  v\cdot \theta^{-i} \otimes n r \otimes \theta^{i}\cdot m\\
        &\quad\: =  \theta^{i}\cdot v\otimes n \otimes r\theta^{i}\cdot m\\
    \end{align*}
    On the other hand, for the term on the right in the equation (\ref{Equ:LeftHandSide}), we have
    \begin{equation}
        n\otimes \theta^i\cdot  v\otimes r\theta^i\cdot m \mapsto \theta^i\cdot  v\otimes n\otimes r\theta^i\cdot m\nonumber
    \end{equation}
    Therefore, the result follows.
\end{proof} 

\subsubsection{Proof of theorem \ref{Thm:EquiFunct}}\label{SSec:ProThmEquFunct}
In this section, we prove the theorem \ref{Thm:EquiFunct}, which is restated below. The functoriality of Khovanov homology is stated in section \ref{SSec:KhovFunct}. In this section and following sections by $\simeq$, we mean chain homotopy up to a $\pm$ sign.

\begin{proof}[Proof of theorem \ref{Thm:EquiFunct}]
From theorem \ref{Thm:EquiMM}, given two equivariant cobordisms $\mathcal{E},\mathcal{E}':L_0\rightarrow L_1$ that are equivariantly isotopic relative to the boundary, they are equivalent up to finitely many equivariant movie moves (definition \ref{Def:EquiMM}). Hence, we have to prove that the maps induced by two equivariant movies that are related by an equivariant movie move, are chain homotopic up to a $\pm$ sign. Without loss of generality, we can assume that $\mathcal{E}$ and $\mathcal{E}'$ differ only by one equivariant movie move. Denote by $\boldsymbol{EM}_{\mathcal{E}}=\{E_i\}_{i=0}^k$ and $\boldsymbol{EM}_{\mathcal{E}'}=\{E'_i\}_{i=0}^k$ the equivariant movie of $\mathcal{E}$ and $\mathcal{E}'$ respectively. Hence, $E_0=E'_0$, and we have a diskular decomposition $E_0=R\circ(T_0,...,T_{p-1})$ for admissible diskular tangles $R\in\mathbb{T}^{(n,\cdots,n;0)}$ and $T_1,...,T_p\in\mathbb{T}^{(;n)}$. Also, $E_k=E'_k$ and we have $E_k=R\circ(S_0,...,S_{p-1})$ for admissible diskular tangles $R\in\mathbb{T}^{(n,\cdots,n;0)}$ and $S_1,...,S_p\in\mathbb{T}^{(;n)}$. Moreover, $(T_0,...,T_{p-1})$ and $(S_0,...,S_{p-1})$ are $p$-equivariant $(;n)$-diskular tangles ($\theta(S_i)=S_{i+1}$ and $\theta(T_j)=T_{j+1}$ for $i,j\in \mathbb{Z}_p=\{0,...,p-1\}$). 

Let us denote $C_i=C_{\text{Kh}}(T_i)$ and $D_i=C_{\text{Kh}}(S_i)$. Both $(T_0,...,T_{p-1})$ and $(S_0,...,S_{p-1})$ being $p$-equivariant $(;n)$-diskular tangles implies that $C_0\cong ... \cong C_{p_1}$ and  $D_0\cong ... \cong D_{p_1}$. Therefore, we denote $C=C_0=...=C_{p-1}$ (respectively $D=D_0=...=D_{p-1}$) and we can write the equivariant Khovanov chain complexes for $E_0$ and $E_k$ by

\begin{equation}\label{Equ:TangDecoFrame0}
    C_{\text{EKh}}(E_0)=E\mathbb{Z}_p^*\underset{\mathbb{Z}[\mathbb{Z}_p]}{\otimes}\big(C_{\text{Kh}}(R\circ(T_0,...,T_0)\big)=E\mathbb{Z}_p^*\underset{\mathbb{Z}[\mathbb{Z}_p]}{\otimes}\big( C_{\text{Kh}}(R)\underset{\mathcal{H}^n\otimes \cdots \otimes \mathcal{H}^n}{\otimes}C^{\otimes p}\big)
\end{equation}
and 
\begin{equation}\label{Equ:TangDecoFrame1}
    C_{\text{EKh}}(E_k)=E\mathbb{Z}_p^*\underset{\mathbb{Z}[\mathbb{Z}_p]}{\otimes}\big( C_{\text{Kh}}(R\circ(S_0,...,S_{p-1})\big)= E\mathbb{Z}_p^*\underset{\mathbb{Z}[\mathbb{Z}_p]}{\otimes}\big(C_{\text{Kh}}(R)\underset{\mathcal{H}^n\otimes \cdots \otimes \mathcal{H}^n}{\otimes}D^{\otimes p}\big)
\end{equation}

Given that $E\mathbb{Z}_p^*$ is isomorphic to $E\mathbb{Z}_p$ as $\mathbb{Z}[\mathbb{Z}_p]$-module, we will drop the dual sign $(\cdot^*)$ from the notation. The chain maps induced by the $\mathcal{E}$ and $\mathcal{E}'$ respectively are given by 
\begin{equation}
C_{\text{EKh}}(\mathcal{E})=id_{C_{\text{Kh}}(R)}\otimes C_{\text{Kh}}(\mathfrak{f},...,\mathfrak{f})\;\text{ and }\;C_{\mathrm{EKh}}(\mathcal{E}')=id_{C_{\text{Kh}}(R)}\otimes C_{\text{Kh}} (\mathfrak{f}',...,\mathfrak{f}'), 
\end{equation}
where $\mathfrak{f},\mathfrak{f}':T_i\rightarrow S_i$ for $i\in\mathbb{Z}_p$ denote the composition of elementary cobordisms localized to diskular tangles $S_i$ and $T_i$ in the frames of the movie of $\mathcal{E}$ and $\mathcal{E}'$ respectively. Also let us donate $f_0=C_{\text{Kh}} (\mathfrak{f})$ and $f_1=C_{\mathrm{Kh}} (\mathfrak{f}')$. Our goal is to show the induced maps $C_{\text{EKh}}(\mathcal{E})$ and $C_{\text{EKh}}(\mathcal{E}')$ are equivariantly chain homotopic, possibly up to a sign.

The chain complexes $C^{\otimes p}$ and $D^{\otimes p}$ are $\mathcal{R}^n_{\theta}$-modules (see section \ref{SSec:TwiGroRin}). The action of $\theta$ commutes with $f_0^{\otimes p}$ and $f_1^{\otimes p}$. Hence, $f_0^{\otimes p}$ and $f_1^{\otimes p}$ are also $\mathcal{R}^n_{\theta}$-module morphisms. Therefore, we need to construct a $\mathcal{R}^n_{\theta}$-chain homotopy between $f_0^{\otimes p}$ and $f_1^{\otimes p}$.

We have the following commutative diagram in $\textbf{Ch}_{\mathcal{R}^{n}_{\theta}}$, the category of equivariant chain complexes:
    \begin{equation}
        \begin{tikzcd}
            \text{Hom}_{\mathcal{K}}(E\mathbb{Z}_p\otimes C^{\otimes p}, D^{\otimes p}) \arrow[d] &\text{Hom}_{\mathcal{K}}(E\mathbb{Z}_p\otimes C^{\otimes p}, E\mathbb{Z}_p\otimes D^{\otimes p})\arrow[d]\arrow[l]\\
            \text{Hom}_{\mathcal{D}}(E\mathbb{Z}_p\otimes C^{\otimes p}, D^{\otimes p}) & \text{Hom}_{\mathcal{D}}(E\mathbb{Z}_p\otimes C^{\otimes p}, E\mathbb{Z}_p\otimes D^{\otimes p})\arrow[l]
        \end{tikzcd}
    \end{equation}
Here, $\mathcal{K}$ (respectively $\mathcal{D}$) denotes the homotopy category (respectively derived category) of $R^n_{\theta}$-modules. 

Since $E\mathbb{Z}_p$ is a free resolution of $\mathbb{Z}$ over $\mathbb{Z}[\mathbb{Z}_p]$-modules, there is a chain map $E\mathbb{Z}_p\to \mathbb{Z}$. If we give $E\mathbb{Z}_p\otimes D^{\otimes p}$ and $\mathbb{Z}\otimes D^{\otimes p}$ the semi-diagonal action by $\mathcal{R}^n_\theta$ (from section \ref{SSec:TwiGroRin}), there is an induced homomorphism of $\mathcal{R}^n_{\theta}$-modules $E\mathbb{Z}_p\otimes D^{\otimes p}\to \mathbb{Z}\otimes D^{\otimes p}\cong D^{\otimes p}$. The top and bottom horizontal arrows are induced by post-composing with this map.

%Since $E\mathbb{Z}_p$ is the free resolution of $\mathbb{Z}$ as $\mathbb{Z}[\mathbb{Z}_p]$-module by semi-diagonal action of $\mathcal{R}^n_{\theta}$ (\ref{Equ:SemiDiagAct}), we have a $\mathcal{R}^n_{\theta}$-module morphism $E\mathbb{Z}_p\otimes D^{\otimes p}\rightarrow \mathbb{Z}\otimes D^{\otimes p}\cong D^{\otimes p}$ (by section \ref{SSec:TwiGroRin}). The top and bottom horizontal arrows are induced by post composing with this map.

By theorem \ref{Thm:KhoTwiProj}, $E\mathbb{Z}_p\otimes C^{\otimes p}$ is a projective, bounded above $\mathcal{R}^n_{\theta}$-chain complex. Hence, lemma \ref{Lem:HomAndDer} implies that both of the vertical maps are isomorphisms. Moreover, $E\mathbb{Z}_p$ is a projective resolution of $\mathbb{Z}$ as a $\mathbb{Z}[\mathbb{Z}_p]$-module. Therefore, $E\mathbb{Z}_p\otimes D^{\otimes p}$ is quasi-isomorphic to $D^{\otimes p}$ as a $\mathcal{R}^n_{\theta}$-chain complex, so the bottom horizontal map is also an isomorphism. Therefore, the top horizontal map is also an isomorphism and the map $\widehat{K}$ defined in proposition \ref{Pro:EquMapK}, corresponds to an $\mathcal{R}^n_{\theta}$-chain homotopy $\tilde{k}$ between chain maps $f_0^{\otimes p}, f_1^{\otimes p} :E\mathbb{Z}_p\otimes C^{\otimes p}\rightarrow E\mathbb{Z}_p\otimes D^{\otimes p}$. 

Now taking the tensor product with $C_{\text{Kh}}(R)$ as a right $\mathcal{R}^n_{\theta}$-module, and extending the chain maps $f_0^{\otimes p}$ and $f_1^{\otimes p}$ by the identity on $C_{\text{Kh}}(R)$, we have
\begin{align}\label{Equ:WrongOrder}
    \tilde{f}_i=id_{C_{\text{Kh}}(R)}\underset{\mathcal{R}^n_{\theta}}{\otimes}f_i^{\otimes p} &:C_{\text{Kh}}(R)\underset{\mathcal{R}^n_{\theta}}{\otimes} (E\mathbb{Z}_p\otimes C^{\otimes p}) \rightarrow C_{\text{Kh}}(R)\underset{\mathcal{R}^n_{\theta}}{\otimes} (E\mathbb{Z}_p\otimes D^{\otimes p}),
\end{align}
for $i=0,1$. To complete the proof we need to re-write the modules in (\ref{Equ:WrongOrder}) similar to the equations (\ref{Equ:TangDecoFrame0}) and (\ref{Equ:TangDecoFrame1}).

Now by the lemma \ref{Lem:SwitchTensor}, we can write the chain complexes as follows.
\begin{align}
    \tilde{f}_i &:E\mathbb{Z}_p\otimes (V\otimes_{(\mathcal{H}^n)^{\otimes p}}C^{\otimes p}) \rightarrow E\mathbb{Z}_p\otimes(V\otimes_{(\mathcal{H}^n)^{\otimes p}} D^{\otimes p}).
\end{align}

The maps $f_0$ and $f_1$ that are induced by the before and the after movie of a movie move are chain homotopic up to a sign \cite{Kho06, Bar05, Jac04}. Hence $f_0^{\otimes p}$ and $f_1^{\otimes p}$ are equivariant chain homotopic up to a factor of $(-1)^p$. In conclusion, 
\begin{equation}
    C_{\text{EKh}}(\Sigma)=id_{C_{\text{Kh}}(R)}\otimes C_{\text{Kh}}(\mathfrak{f},...,\mathfrak{f}) =id_{C_{\text{Kh}}(R)}\otimes f_0^{\otimes p}
\end{equation}
and 
\begin{equation}
    C_{\text{EKh}}(\Sigma')=id_{C_{\text{EKh}}(R)}\otimes C_{\text{Kh}} (\mathfrak{f}',...,\mathfrak{f}')=id_{C_{\text{Kh}}(R)}\otimes f_1^{\otimes p}
\end{equation}
are equivariantly chain homotopic up to $(-1)^p$.
\end{proof}

\section{Obstruction to equivariant concordance}\label{Sec:EquConcObs}
This section is devoted to proving theorem \ref{Thm:EquiRibbCon}. Some of the notations are similar to the section \ref{SSec:NecCutRibbon}.

\subsection{Equivariant neck cutting relation}\label{SSec:EquiNecCut}

Let $\text{K}ob^2_{p,\bullet}\subset\mathcal{K}ob^2_{\bullet}$ denote the $\mathbb{Z}$-linearized category with objects the collection of $p$-periodic links. The morphisms of $\text{K}ob^2_{p,\bullet}$ consist of finite formal sum of cobordisms (surfaces) in $\mathbb{R}^3\times[0,1]$ with a choice of finitely many dots on the cobordisms away from their boundaries. These cobordisms are not a priori equivariant under the extended action of $\tilde{\theta}$ but $\mathbb{Z}_p$ acts on formal sums of cobordism by $\tilde{\theta}$ extended linearly. We denote by $(\text{K}ob^2_{p,\bullet})^{\mathbb{Z}_p}$ the fixed set of the action $\tilde{\theta}$.

Additionally, let $\mathbb{Z}_p\mathcal{K}ob^2$ denote the $\mathbb{Z}$-linear category with objects the $p$-periodic links in $\mathbb{R}^3\backslash\boldsymbol{z}$. Morphisms are the formal linear combination of $p$-equivariant cobordisms in $(\mathbb{R}^3\backslash\boldsymbol{z})\times[0,1]$ (definition \ref{Def:EquiCob}). 

\begin{Definition} 
Assume $\Sigma$ is a $p$-equivariant cobordism, and $h_i$ are smooth, embedded 1-handles $[-1,1]\times\mathbb{D}^2\rightarrow \mathbb{R}^3\times [0,1]$. By a \emph{standard} $p$-\textit{equivariant neck} $\tilde{\mathfrak{n}}=(\mathfrak{n}_1,...,\mathfrak{n}_p)$ on $\Sigma$, we are referring to a collection of $p$ disjoint standard necks $\mathfrak{n_i}=h_i\cap \Sigma$ (see definition \ref{Def:Neck}), such that
\begin{enumerate}
    \item[(EN1)] The $h_i$ are permuted by the action $\tilde{\theta}$
    \item[(EN2)] The $h_i$ are disjoint from the $\boldsymbol{z}\times[0,1]$. 
\end{enumerate}
\end{Definition}

We can define \textit{equivariant neck cutting} as follows. First, for an equivariant cobordism $\Sigma\in \mathbb{Z}_p\mathcal{K}ob^2$, we remove the necks $\mathfrak{n}_i=[-1,1]\times S_i^1\hookrightarrow\Sigma$ for $i=1,...,p$. Then for $i=1,...,p$, we glue disks $\{-1,1\}\times \mathbb{D}_i^2$ to the cutout boundaries obtained from cutting $\mathfrak{n}_i$. Lastly, for each $1\leq i\leq p$, we either place a dot on $\{1\}\times \mathbb{D}_i^2$ or on $\{-1\}\times \mathbb{D}_i^2$ and denote the resulting dotted cobordism by $\mathcal{E}_{(\iota_1,...,\iota_{i-1},\boldsymbol{+},\iota_{i+1},...,\iota_p)}$ if the dot is on  $\{1\}\times \mathbb{D}_i^2$ or respectively by $\mathcal{E}_{(\iota_1,...,\iota_{i-1},\boldsymbol{-},\iota_{i+1},...,\iota_p)}$ if the dot is placed on $\{-1\}\times \mathbb{D}_i^2$. Here $\iota_1,...,\iota_p\in \{+,-\}$. To wit, the result of equivariant neck cutting on an standard equivariant neck $\tilde{\mathfrak{n}}$ on equivariant cobordism $\Sigma$ can be understood as a map that sends $\Sigma \in \mathbb{Z}_p\mathcal{K}ob^2$ to the formal sum
\begin{equation}\label{Equ:EquNecCut2}
    \mathcal{E}=\sum_{(\iota_1,...,\iota_p)\in \{\boldsymbol{+},\boldsymbol{-}\}^p} \mathcal{E}_{(\iota_1,...,\iota_p)}\in (\text{K}ob^2_{p,\bullet})^{\mathbb{Z}_p}.
\end{equation}

\begin{figure}
    \centering
	\def\svgwidth{1\textwidth}
	%% Creator: Inkscape 1.2.2 (b0a8486, 2022-12-01), www.inkscape.org
%% Accompanies image file 'NecCutCobo.pdf' (pdf, eps, ps)
\begingroup%
  \makeatletter%
  \providecommand\color[2][]{%
    \errmessage{(Inkscape) Color is used for the text in Inkscape, but the package 'color.sty' is not loaded}%
    \renewcommand\color[2][]{}%
  }%
  \providecommand\transparent[1]{%
    \errmessage{(Inkscape) Transparency is used (non-zero) for the text in Inkscape, but the package 'transparent.sty' is not loaded}%
    \renewcommand\transparent[1]{}%
  }%
  \providecommand\rotatebox[2]{#2}%
  \newcommand*\fsize{\dimexpr\f@size pt\relax}%
  \newcommand*\lineheight[1]{\fontsize{\fsize}{#1\fsize}\selectfont}%
  \ifx\svgwidth\undefined%
    \setlength{\unitlength}{566.55058132bp}%
    \ifx\svgscale\undefined%
      \relax%
    \else%
      \setlength{\unitlength}{\unitlength * \real{\svgscale}}%
    \fi%
  \else%
    \setlength{\unitlength}{\svgwidth}%
  \fi%
  \global\let\svgwidth\undefined%
  \global\let\svgscale\undefined%
  \makeatother%
  \begin{picture}(1,0.13512807)%
    \lineheight{1}%
    \setlength\tabcolsep{0pt}%
    \put(0,0){\includegraphics[width=\unitlength,page=1]{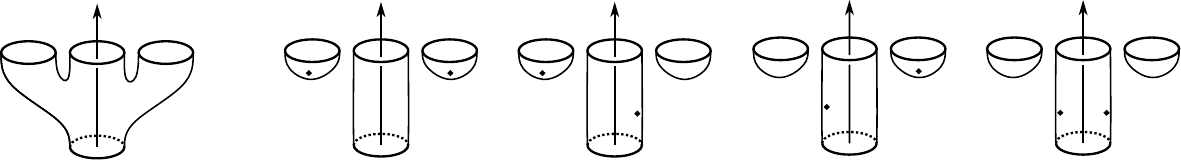}}%
    \put(0.42490274,0.03499618){\color[rgb]{0,0,0}\makebox(0,0)[t]{\lineheight{1.25}\smash{\begin{tabular}[t]{c}$+$\end{tabular}}}}%
    \put(0.62347299,0.03499618){\color[rgb]{0,0,0}\makebox(0,0)[t]{\lineheight{1.25}\smash{\begin{tabular}[t]{c}$+$\end{tabular}}}}%
    \put(0.82204346,0.03499618){\color[rgb]{0,0,0}\makebox(0,0)[t]{\lineheight{1.25}\smash{\begin{tabular}[t]{c}$+$\end{tabular}}}}%
    \put(0.19985667,0.03499618){\color[rgb]{0,0,0}\makebox(0,0)[t]{\lineheight{1.25}\smash{\begin{tabular}[t]{c}$\mapsto$\end{tabular}}}}%
  \end{picture}%
\endgroup%

    \caption[Equivariant neck cutting]{An example of the 2-Equivariant neck cutting. The dotted cobordism depicted on the right hand side is $\mathcal{E}_{(\boldsymbol{+},\boldsymbol{+})}+\mathcal{E}_{(\boldsymbol{+},\boldsymbol{-})}+\mathcal{E}_{(\boldsymbol{-},\boldsymbol{+})}+\mathcal{E}_{(\boldsymbol{-},\boldsymbol{-})}$.}
    \label{Fig:EquiNecCutCob}
\end{figure}

For instance, in the notation above, the $\mathcal{E}_{(\boldsymbol{+},\cdots,\boldsymbol{+})}$ indicates that after all of the necks $\mathfrak{n}_i$ were removed and we placed a dot on $\{+1\}\times \mathbb{D}^2$ for each $1\leq i\leq p$. More generally, after equivariant neck cuts on a $p$-equivariant cobordism $\Sigma$, we have $2^p$ choices for placements of dots on the glued disks $(\{\pm1\}\times\mathbb{D}^2)_{i=1}^p$. Figure \ref{Fig:EquiNecCutCob} shows the equivariant neck cutting on a $2$-equivariant cobordism $\Sigma$.  

\begin{Proposition}\label{Pro:NecCutEquKhov}
    Assume $\Sigma$ is an equivariant cobordism and $\mathcal{E}$ denotes the equivariant neck cutting of $\Sigma$ (\ref{Equ:EquNecCut2}). Then, equivariant Khovanov homology satisfies the equivariant neck cutting relation i.e.,
    \begin{equation}\label{Equ:NecCutSum}
        C_{\text{EKh}}(\Sigma) \simeq \sum_{[\iota]\in \{+,-\}/\mathbb{Z}_p}C_\text{EKh}\big(\sum_{(\iota_1,...,\iota_p)\in [\iota]}\mathcal{E}_{(\iota_1,...,\iota_p)}\big)
    \end{equation}
    where $\simeq$ chain homotopy up to a sign as was mentioned earlier in this section, and  $[\iota]$ denotes the class of $\iota\in \{+,-\}^p$ modulo the action of $\mathbb{Z}_p$ by cyclic permutation on $\{+,-\}^p$.
\end{Proposition}
\begin{proof}
    For simplicity, we can assume that $\Sigma$ consists of only one standard equivariant neck $\tilde{\mathfrak{n}}=(\mathfrak{n}_i)_{i=1}^p$ such that the equivariant movie of $\Sigma$ is given by considering $p$ copies of the movie in figure \ref{Fig:NeckCut} (a). The general case follows from the simple case inductively. Let $E_i=R\circ(T_{1,i},\cdots,T_{p,i})$ for $i=0,1,2$ denotes the equivariant movie of the $\Sigma$, where $R$ is a $(n,\cdots,n;0)$-diskular tangle and for $i=1,...,p$, diskular tangles $T_{0,i},T_{1,i}$ and $T_{2,i}$ are the $(;4)$-diskular tangle in figure \ref{Fig:NeckCut}(a) from left to right respectively. The elementary cobordism $\mathfrak{f}_i:T_{0,i}\rightarrow T_{2,i}$ (definition \ref{Def:SimpCob}) for $i=1,...,p$ consists of two elementary saddle moves.
    
    There are induced chain maps $C_{\text{Kh}}(id_{R})=id_{C_{\text{Kh}}(R)}$, and 
    $$\bigotimes_{i=1}^pC_{\text{Kh}}(\mathfrak{f}_i): C_{\text{Kh}}(T_{0,1})\otimes ...\otimes C_{\text{Kh}}(T_{0,p})\rightarrow C_{\text{Kh}}(T_{2,1})\otimes ...\otimes C_{\text{Kh}}(T_{2,p}),$$
    where $C_{\text{Kh}}(R)$ and $C_{\text{Kh}}(T_{j,i},...,T_{j,i})$ are $\mathcal{R}^n_{\theta}$-modules. Because Khovanov homology satisfies the neck cutting relation we can write 
    \begin{align}
        C_{\text{Kh}}(\Sigma) &= \big(C_{\text{Kh}}(\mathcal{E}_{1,+})+C_{\text{Kh}}(\mathcal{E}_{1,-})\big)\otimes\cdots\otimes \big(C_{\text{Kh}}(\mathcal{E}_{p,+})+C_{\text{Kh}}(\mathcal{E}_{p,-})\big)\\
        &=\sum_{(\iota_1,...,\iota_p)\in \{\boldsymbol{+},\boldsymbol{-}\}^p} \bigotimes_{j=1}^p C_{\text{Kh}}(\mathcal{E}_{\iota_j}) \nonumber\\
        &=\sum_{(\iota_1,...,\iota_p)\in \{\boldsymbol{+},\boldsymbol{-}\}^p} C_{\text{Kh}}(\mathcal{E}_{(\iota_1,...,\iota_p)})\nonumber
    \end{align}
    where $C_{\text{Kh}}(\mathcal{E}_{i,+})+C_{\text{Kh}}(\mathcal{E}_{i,-})$ denotes the result of neck cutting on the neck $\mathfrak{n}_i$ for $i=1,...,p$. Also, the notation $\mathcal{E}_{\iota_j}$ denotes $\Sigma$ after the all of the $p$ necks have replaced by dotted disks according to $\iota_j\in\{+,-\}^p$. 
    By the action of $\mathbb{Z}_p$ on $\{+,-\}^p$, for each $[\iota']\in\{+,-\}^p/\mathbb{Z}_p$, the map induced by $\sum_{\iota\in[\iota']}C_{\text{Kh}}(\mathcal{E}_{\iota})$ is equivariant. Therefore, we can re-write the sum above over the equivalence classes $\{+,-\}^p/\mathbb{Z}_p$ and we get the equation  (\ref{Equ:NecCutSum}), as desired.
\end{proof}

\subsection{Map induced by equivariant ribbon concordance}\label{SSec:EquRibConMap}
Now we can prove theorem \ref{Thm:EquiRibbCon}. Our proof mirrors that of \cite{LZ19}.

\begin{Definition}\label{Def:EquiConc}
    An \emph{equivariant ribbon concordance} from a $p$-periodic knot $K_0$ to a $p$-periodic knot $K_1$ is a smoothly embedded annulus $F\subset (\mathbb{R}^3\backslash\boldsymbol{z}) \times [0,1]$ such that the projection $\rho: (\mathbb{R}^3\backslash\boldsymbol{z}) \times [0,1]\rightarrow [0,1]$ restricted to $F$ is a Morse function. Moreover,
    \begin{enumerate}
        \item It is invariant under the extended action $\tilde{\theta}$ on $S^3\times [0,1]$.
        \item It has only index 0 and 1 critical points.
    \end{enumerate}
\end{Definition}
    
\begin{proof}[Proof of theorem \ref{Thm:EquiRibbCon}] 
By theorem \ref{Thm:EquiFunct}, we can equivariantly isotope $F:K_0\rightarrow K_1$ such that the movie of $F$ has the following order. The index $0$ critical points (births) appear first, followed by a sequence of equivariant planer isotopy and equivariant Reidemeister moves. Lastly, $F$ has index 1 critical points (saddles), where the saddles can be viewed as attaching unknotted equivariant bands between two strands of the link diagram. 

By $F^{\text{op}}$ we denote the ribbon concordance $F$ with the opposite orientation. Gluing $F$ along $K_1$ to $F^{\text{op}}$ on $K_1$ results in the cobordism $F^{\text{op}}\circ F$ from $K_0$ to itself. Both $F$ and $F^{\text{op}}$ are equivariant. Hence, $\Sigma= F^{\text{op}}\circ F$ is an equivariant cobordism. Then $\Sigma$ can be structured as a finite collection of equivariant unknotted 2-spheres $\{S^2_j\}$ (obtained from index 0 and 2 critical points of $\Sigma$) tubed to annulus $K_0\times [0,1]$ by a finite collection of standard equivariant necks (obtained from the index 1 critical points of $F$ and $F^{\text{op}}$)
%\begin{itemize}
 %   \item A finite collection of equivariant unknotted 2-spheres $\{S^2_j\}$ obtained by attaching the 2-disks obtained from index 0 critcal points of $F$ to the index 2 critical points of $F^{\text{op}}$. The spheres could be linked to the annulus $K_0\times [0,1]$.
%    \item A finite collection of standard equivariant necks obtained by the index 1 critical points of $F$ and $F^{\text{op}}$ with one end on the spheres above and the other end on the $K_0\times [0,1]$.
%\end{itemize}

For simplicity of notation, we assume that $F$ has only one equivariant index 0 critical point (which consists of $p$ distinct index 0 critical points that correspond by the action $\tilde{\theta}$), and only one equivariant index 1 critical point. The general case would follow by applying the same argument inductively to a general collection of equivariant necks and spheres.

We can perform equivariant neck cutting to eliminate the equivariant neck. Let $\mathcal{E}$ denote the resulting cobordism, which viewed as an element of $(\mathcal{K}ob_{\bullet}^2)^{\mathbb{Z}_p}$ is a finite formal sum of cobordisms consisting of finitely many equivariant embedded, unknotted 2-spheres with dots $(S^2,\cdots,S^2)$, and an embedded cylinder $C'$ with dots and not linked with the $(S^2,\cdots,S^2)$ above. 
\begin{equation}\label{Equ:EquNecCut}
    \mathcal{E}= \sum_{(\iota_1,...,\iota_p)\in\{\bullet,\circ\}^p}C'_{(\overline{\iota}_1,...,\overline{\iota}_p)}\cup (S^2_{\iota_1},...,S^2_{\iota_p})
\end{equation}
where $\overline{\iota}_j=\bullet$ if $\iota_j=\circ$ and, $\overline{\iota}_j=\circ$ if $\iota_j=\bullet$. Also, $S^2_{\iota_j}$ is a dotted sphere if $\iota_j=\bullet$ and is a sphere with no dots if $\iota_j=\circ$.

Not all of the terms in the equation \ref{Equ:EquNecCut} are individually equivariant dotted cobordisms. In the formal sum above, the cobordism $C'_{\bullet,\cdots,\bullet}\cup(S^2_{\circ},\cdots, S^2_{\circ})$ which has all of the dots on $C'$ is an equivariant dotted cobordism. Similarly, the cobordism $C'_{\circ,\cdots,\circ}\cup(S^2_{\bullet},\cdots, S^2_{\bullet})$ which has all of the dots on the spheres is also an equivariant dotted cobordism. The other terms are not taken to themselves by the action of $\tilde{\theta}$. Moreover, the extended action $\tilde{\theta}$ acts on $(S^2_{\iota_1},...,S^2_{\iota_n})$ by cyclic permutation of the factors. Let $[\iota]\in\{\bullet,\circ\}^p/\mathbb{Z}_p$ denote the orbits of the action of $\mathbb{Z}_p$ on $\{\bullet,\circ\}^p$ by cyclic permutation. Then, for each $[\iota]$,
\begin{align}\label{Equ:EquiSum}
\mathcal{E}_{[\iota]} &=\sum_{(\iota_1,...,\iota_p)\in [\iota]}C'_{(\overline{\iota}_1,...,\overline{\iota}_p)}\cup (S^2_{\iota_1},...,S^2_{\iota_p})
\end{align}
is fixed by $\tilde{\theta}$, so is an equivariant cobordism.

The dotted cobordism $C'_{\bullet,\cdots,\bullet}\cup(S^2_{\circ},\cdots, S^2_{\circ})$ (respecively $C'_{\circ,\cdots,\circ}\cup(S^2_{\bullet},\cdots, S^2_{\bullet})$) is equivariantly isotopic to $C_{\bullet,\cdots,\bullet}\cup(S^2_{\circ},\cdots, S^2_{\circ})$ (respectively $C_{\circ,\cdots,\circ}\cup(S^2_{\bullet},\cdots, S^2_{\bullet})$), where $C_{{\bullet,\cdots,\bullet}}$ (respectively $C_{\circ,\cdots,\circ}$) denotes the standard cylinder $K_0\times [0,1]$ with $p$ dots (respectively with no dots) which is split from the spheres $(S^2_{\circ},\cdots, S^2_{\circ})$ (respectively $(S^2_{\bullet},\cdots, S^2_{\bullet})$). We have
\begin{align}\label{Equ:LastThmE1}
    C_{\text{EKh}}(C'_{\bullet,\cdots,\bullet}\cup(S^2_{\circ},\cdots, S^2_{\circ})) &\simeq
    C_{\text{EKh}}(C_{\bullet,\cdots,\bullet}\cup(S^2_{\circ},\cdots, S^2_{\circ}))\\
    &= id_{E\mathbb{Z}_p}\otimes_{\mathbb{Z}[\mathbb{Z}_p]}C_{\text{Kh}}(C_{\bullet,\cdots,\bullet} \sqcup (S^2_{\circ},...,S^2_{\circ}))= 0,\nonumber
\end{align}
where the first chain homotopy is by theorem \ref{Thm:EquiFunct}. Note that the induced map $C_{\text{Kh}}(S^2_{\circ}):\mathbb{Z}\rightarrow \mathbb{Z}$ by a sphere with no dots is zero, and the map $C_{\text{Kh}}(C_{\bullet,\cdots,\bullet})$ induced by a cobordism with more than one dot is null homotopic. Also,
\begin{align}\label{Equ:LastThmE2}
    C_{\text{EKh}}(C'_{\circ,\cdots,\circ}\cup(S^2_{\bullet},\cdots, S^2_{\bullet})) &\simeq C_{\text{EKh}}(C_{\circ,\cdots,\circ}\cup(S^2_{\bullet},\cdots, S^2_{\bullet})).\\
    &= id_{E\mathbb{Z}_p}\otimes_{\mathbb{Z}[\mathbb{Z}_p]} C_{\text{Kh}}(C_{\circ,\cdots,\circ}\sqcup (S^2_{\bullet},...,S^2_{\bullet})).\nonumber\\
    &= id_{E\mathbb{Z}_p}\otimes_{\mathbb{Z}[\mathbb{Z}_p]}\bigg( id_{C_{\text{Kh}}(K_0)} \otimes (C_{\text{Kh}}(S^2_{\bullet})\otimes ...\otimes C_{\text{Kh}}(S^2_{\bullet}))\bigg).\nonumber\\
    &= id_{C_{\text{EKh}}(K_0)}\nonumber.
\end{align}
This follows because the map $C_{\text{Kh}}(S^2_{\bullet}):\mathbb{Z}\rightarrow \mathbb{Z}$ induced by a sphere with one dot is multiplication by $1\in \mathbb{Z}$.

For a fixed $[\eta]\in \{\bullet,\circ\}^p/\mathbb{Z}_p-\{[\bullet,\cdots,\bullet],[\circ,\cdots,\circ]\}$ and $\iota\in[\eta]$, let $f_s$ for $s\in [0,1]$ denote the ambient isotopy relative to the boundary that transforms $C'_{(\overline{\iota}_1,...,\overline{\iota}_p)}\cup (S^2_{\iota_1},...,S^2_{\iota_p})$ to an $\iota$-dotted cobordism $C_{(\overline{\iota}_1,...,\overline{\iota}_p)}\cup (S^2_{\iota_1},...,S^2_{\iota_p})$ with $C$ the standard cylinder $K_0\times[0,1]$ and disjoint from $(S^2_{\iota_1},...,S^2_{\iota_p})$. We can equivariantly isotope the equivariant cobordisms (\ref{Equ:EquiSum}) by $\tilde{f}_s=\sum_{\theta\in\mathbb{Z}_p}\hat{\theta}^{-1} f_s \hat{\theta}$. The $\tilde{f}_s:\mathcal{E}_{[\iota]}\rightarrow \mathcal{E}_{[\iota]}$ is an equivariant map with respect to the extended action $\tilde{\theta}$ for all $s\in[0,1]$. Therefore, $\tilde{f}_s$ induces a $\mathbb{Z}[\mathbb{Z}_p]$-chain homotopy 
\begin{equation}
    C_{\text{EKh}}(\mathcal{E}_{[\iota]})\simeq C_{\text{EKh}}(\mathcal{C}_{[\iota]}),
\end{equation}
where
\begin{equation}
    \mathcal{C}_{[\eta]}=\sum_{(\iota_1,...,\iota_p)\in [\eta]}C_{(\overline{\iota}_1,...,\overline{\iota}_p)}\cup (S^2_{\iota_1},...,S^2_{\iota_p}).
\end{equation}

Because the chain maps induced on the Khovanov chain complex by a cobordism consisting of more than one dot on a connected component, and the cobordism consisting of a 2-sphere with no dots as a connected component are nullhomotopic, we have
\begin{align*}
C_{\text{EKh}}(\mathcal{E}_{[\iota]})&\simeq C_{\text{EKh}}(\mathcal{C}_{[\iota]})\\
&= C_{\text{EKh}}(\sum_{(\iota_1,...,\iota_p)\in [\eta]}C_{(\overline{\iota}_1,...,\overline{\iota}_p)}\cup (S^2_{\iota_1},...,S^2_{\iota_p})\simeq 0\nonumber
\end{align*}
for all $\eta\neq (\bullet,\cdots,\bullet),(\circ,...,\circ)$. Hence we have
\begin{align*}
C_\text{EKh}(F^{op}\circ F) &\simeq C_\text{EKh}\big(\sum_{(\iota_1,...,\iota_p)\in\{\bullet,\circ\}^p}C'_{(\overline{\iota}_1,...,\overline{\iota}_p)}\cup (S^2_{\iota_1},...,S^2_{\iota_p}) \big)\\
     &\simeq \sum_{[\iota]\in \{\bullet,\circ\}/\mathbb{Z}_p} C_\text{EKh}\big(\sum_{(\iota_1,...,\iota_p)\in [\iota]}C_{(\overline{\iota}_1,...,\overline{\iota}_p)}\cup (S^2_{\iota_1},...,S^2_{\iota_p})\big)\\
     &=C_\text{EKh}(C_{\circ,\cdots,\circ}\cup(S^2_{\bullet},\cdots, S^2_{\bullet}))=id_{C_\text{EKh}(L_0)}\\
\end{align*}
This finishes the proof.
\end{proof}

\printbibliography
\end{document}